\documentclass[10pt,reqno]{amsart}

\usepackage[a4paper,left=30mm,right=30mm,top=25mm,bottom=25mm,marginpar=29mm]{geometry}

\usepackage{amsfonts}
\usepackage{amsthm}
\usepackage{amssymb}
\usepackage{amsmath}
\usepackage{amscd}
\usepackage{mathtools}
\usepackage{color}
\usepackage{enumerate}
\usepackage{dsfont}
\usepackage[latin1]{inputenc}
\usepackage[displaymath,mathlines]{lineno}
\usepackage[mathscr]{eucal}
\numberwithin{equation}{section}
\usepackage{epstopdf} 
\usepackage{hyperref}
\usepackage{indentfirst}
\usepackage{lineno}
\setpagewiselinenumbers
 
 \makeindex

\theoremstyle{plain}
\newtheorem{theorem}{Theorem}[section]
\newtheorem{lemma}[theorem]{Lemma}

\newtheorem{proposition}[theorem]{Proposition}

\theoremstyle{definition}
\newtheorem{definition}[theorem]{Definition}

\begin{document}
  \numberwithin{equation}{section}
\title[High-Friction Limit to Bipolar Drift-Diffusion]{The relaxation limit of bipolar fluid models}

\author{Nuno J. Alves}
\address[Nuno J. Alves]{
        King Abdullah University of Science and Technology, CEMSE Division, Thuwal, Saudi Arabia, 23955-6900.}
\email{nuno.januarioalves@kaust.edu.sa}

\author{Athanasios E. Tzavaras}
\address[Athanasios E. Tzavaras]{King Abdullah University of Science and Technology, CEMSE Division, Thuwal, Saudi Arabia, 23955-6900.}
\email{athanasios.tzavaras@kaust.edu.sa}


\date{\today}

\begin{abstract}
This work establishes the relaxation limit from the bipolar Euler-Poisson system to the bipolar drift-diffusion system,
for data so that the latter has a smooth solution.
A relative energy identity is developed for the bipolar fluid system and is used to show that a dissipative weak solution of the bipolar Euler-Poisson system converges in the high-friction regime to a strong and bounded away from vacuum solution of the bipolar drift-diffusion system.

\end{abstract}

\maketitle


\section{Introduction}
This article studies the emergence of the bipolar drift-diffusion system \\
\begin{equation} \label{BDD1}
\begin{dcases}
\rho_t=\nabla \cdot \Big(\rho \nabla\frac{\delta \mathcal{E}}{\delta \rho}\Big) \\
n_t=\nabla \cdot \Big(n\nabla\frac{\delta \mathcal{E}}{\delta n}\Big)\\
 -\Delta \phi = \rho - n
\end{dcases}                                                                                                                                                                                                                                                        \end{equation} 
\begin{equation*}
\rho \nabla \frac{\delta \mathcal{E}}{\delta \rho} \cdot \nu = n\nabla\frac{\delta \mathcal{E}}{\delta n} \cdot \nu = \frac{\partial \phi}{\partial \nu} = 0, \ \ \ \text{on} \ [0,T[ \times \partial \Omega
\end{equation*} \\
as a relaxation limit of the bipolar Euler-Poisson system: \\
\begin{equation} \label{BEP1}
\begin{dcases} 
 \rho_t + \nabla \cdot (\rho u) = 0 \\ 
 (\rho u)_t + \nabla \cdot (\rho u \otimes u) = -\frac{1}{\varepsilon} \rho \nabla \frac{\delta \mathcal{E}}{\delta \rho} -\frac{1}{\varepsilon} \rho u \\ 
 n_t + \nabla \cdot (n v) = 0  \\ 
 (nv)_t + \nabla \cdot (nv \otimes v) = -\frac{1}{\varepsilon} n \nabla\frac{\delta \mathcal{E}}{\delta n} -\frac{1}{\varepsilon} nv \\ 
 -\Delta \phi = \rho - n
 \end{dcases}
\end{equation}
\begin{equation*}
u \cdot \nu = v \cdot \nu = \frac{\partial \phi}{\partial \nu} = 0, \ \ \ \text{on} \  [0,T[ \times \partial \Omega
\end{equation*} \\
in the space-time domain $]0,T[ \times \Omega$, where $T>0$ is a fixed time horizon and $\Omega$ is a smooth bounded domain of $\mathbb{R}^d$ with smooth boundary $\partial \Omega$, where $d \in \mathbb{N} \setminus \{1,2 \}$. 
The systems are expressed using the formalism of Euler flows in \cite{GLT}, based on the functional
\begin{align}
\label{bipen}
\mathcal{E}(\rho,n) \coloneqq \int_\Omega &h_1(\rho)+h_2(n)+ \tfrac{1}{2}|\nabla \phi|^2dx, 
\\
\label{poisson}
-\Delta \phi &= \rho -n.
\end{align}
There is a coupling of the two densities via the Poisson equation \eqref{poisson}, and 
$\dfrac{\delta \mathcal{E}}{\delta \rho}$, $\dfrac{\delta \mathcal{E}}{\delta n} $ stand for the functional derivatives of \eqref{bipen} which are given by
\begin{equation} \label{bipderiv}
\begin{aligned}
\dfrac{\delta \mathcal{E}}{\delta \rho}(\rho,n) = h_1^\prime(\rho) + \phi \, , \quad 
\dfrac{\delta \mathcal{E}}{\delta n}(\rho,n) = h_2^\prime(n) - \phi. 
\end{aligned}
\end{equation}

The models \eqref{BDD1} and \eqref{BEP1} (subject to \eqref{bipen}-\eqref{poisson}) describe two charged fluid systems interacting through 
an electrostatic potential, and are basic models for applications in semiconductor devices or plasma physics, \cite{transportjungel,markowich},\cite[Chapter 3, Subsection 3.3.7]{chenFF}.
Our objective is to describe the
relation between the two models, thus extending the framework of convergence developed in \cite{gasdynamics} for a single fluid system. 
The technical tool for the comparison is a relative energy identity for the bipolar fluid models considered here. 
The functions $h_1,h_2$ represent the internal energies of the fluids, and the electrostatic potential $\phi$ is obtained from the fluid densities $\rho,n$ via the elliptic equation $-\Delta \phi = \rho -n$.  The solution of the Poisson equation is expressed as $$\phi(t,x) = \big(N*(\rho - n)\big)(t,x) \coloneqq \int_\Omega N(x,y) \big(\rho(t,y) - n(t,y) \big)dy,$$ and its spatial gradient is understood as 
$$\nabla \phi(t,x) = \big( \nabla_x N*(\rho - n)\big)(t,x) \coloneqq \int_\Omega \nabla_x N(x,y) \big(\rho(t,y) - n(t,y) \big)dy,$$
where $N$ is the Neumann function \cite{carloskenig}. 
Using the symmetry of $N,$ one derives the formulas \eqref{bipderiv} for the functional derivatives 
$\frac{\delta \mathcal{E}}{\delta \rho}$, $\frac{\delta \mathcal{E}}{\delta n} $. 
Introducing \eqref{bipderiv} to \eqref{BEP1} leads to  \eqref{BEP},
where $p_1,p_2$ are the pressures connected to the internal energies via the usual thermodynamic formulas (\ref{thermoconsistency}).
The formal relaxation limit of \eqref{BEP1} is the system \eqref{BDD1}; establishing this limit is the objective
of the present work.

\par
Relaxation problems arise in physics and chemistry, and from a mathematical viewpoint they have been analysed in several contexts. 
Compensated compactness methods have been used to perform the relaxation limit of single-species hydrodynamic models towards a drift-diffusion equation in one and three  spatial dimensions \cite{natalini,lattanzio}.  We refer to \cite{DD1996} for an interesting analysis leading to existence of weak solutions for the bipolar Euler-Poisson in one-space dimension.
\par

The relative energy method is used here to perform this limiting process for strong solutions of \eqref{BDD1} in several space dimensions. This approach was successful for the relaxation limit in single-species fluid models \cite{gasdynamics,carrillo}, as well as for certain (weakly coupled through friction) multicomponent systems \cite{xiaokai}. 
The relative energy method provides an efficient mathematical mechanism for stability analysis and establishing limiting processes; see \cite{dafermus2} for early developments, \cite{diffusiverelax3,gasdynamics, lattanziothanos2013} and references therein for applications to diffusive relaxation.
Here, the bipolar fluid models are considered in a bounded domain, which is closer to the actual physical situation and requires handling the boundary conditions. No-flux boundary conditions are applied to the velocities for (\ref{BEP1}), to the fluxes for (\ref{BDD1}), and to the electric field for both systems. 

\par 
In order to compare a solution of  (\ref{BEP1}) with a solution of (\ref{BDD1}), one calculates the evolution of  a relative energy functional; 
the formal calculation is presented in subsection \ref{sec3}.  The main convergence result is stated in section \ref{sec4}, and the convergence analysis is 
carried out in section \ref{sec5}.  
This comparison is effected between a dissipative weak solution of (\ref{BEP1}) and a classical and bounded away from vacuum solution of (\ref{BDD1}); the
precise hypotheses are stated in section \ref{sec4}. 
The solution of (\ref{BDD1}) is regarded as an approximate solution of (\ref{BEP1}) and the relative energy identity in the relevant regularity class
is done in Proposition \ref{relativeentropy}.
The technical part amounts to bound the error terms in the relative energy identity. 
The term requiring attention is the one associated with the electric field. Due to the antisymmetry of the electric charges and the fact that the velocities of the fluids are distinct, one cannot reproduce the argument of \cite{gasdynamics} to simplify this electric field term. The desired bound is reached using results on Riesz potentials \cite{stein} 
and Neumann functions \cite{carloskenig}. A Gronwall inequality then yields the relaxation convergence as a stability result. 
The latter is the main result  of this work, Theorem \ref{mainresult},  and shows that if a strong solution of (\ref{BDD1}) is bounded away from vacuum and 
the initial data converge at the initial time then this convergence is preserved for all times $t \in [0,T[.$

\section{Bipolar fluid models}\label{sec:bipolardd}
The systems of equations considered in this article describe the dynamics of fluids formed by charged particles. Such models are common in semiconductor devices (electrons and holes), or in modeling of plasmas, and play a significant role in various technological contexts 
related to semiconductors or plasma physics. Both systems can be derived from the semi-classical bipolar Boltzmann model \cite{transportjungel}. 

We introduce the monotone increasing pressure functions 
$
p_1, p_2 \in C^2(]0, +\infty[) \cap C([0,+\infty[)$ which satisfy  $p_i^\prime(r) > 0$ for $r > 0$ and  $i=1,2,$ and are connected to the internal energy functions $h_1, h_2 \in C^3(]0,+\infty[)\cap C([0,+\infty[)$ through the thermodynamic consistency relations
\begin{equation} \label{thermoconsistency}
r h_i^{\prime \prime}(r)=p_i^{\prime}(r), \quad r h_i^{\prime}(r)=p_i(r)+h_i(r),
\end{equation} 
for $r > 0$ and $i=1,2.$
Observe that $h_i^{\prime \prime}(r) > 0$ for $r>0$ and $i=1,2$ corresponds to the monotonicity of the pressures.
\subsection{Bipolar drift-diffusion}
For the energy \eqref{bipen}-\eqref{poisson} the system (\ref{BDD1}) is composed by two drift-diffusion equations for the densities, coupled with a Poisson equation for the electrostatic potential:
\begin{equation} \label{BDD}
\begin{cases}
\rho_t =\nabla \cdot \big( \nabla p_1(\rho) +\rho \nabla \phi \big) 
\\
n_t = \nabla \cdot \big( \nabla p_2(n) - n \nabla \phi \big)
\\
 -\Delta \phi = \rho - n.
\end{cases}                                                                                                                                                                                                                                                        \end{equation} \par
No-flux boundary conditions are considered for the fluxes and for the gradient of the electrostatic potential,
that is,
\begin{equation} \label{boundarycondBDD}
( \nabla p_1 (\rho) + \rho \nabla  \phi ) \cdot \nu = ( \nabla p_2 (n) - n \nabla  \phi )  \cdot \nu  = \frac{\partial \phi}{\partial \nu}=0 \ \text{on} \ [0,T[ \times \partial \Omega,  \ \ \ \int_{\partial \Omega} \phi \ dx = 0.
\end{equation}

The condition $\int_{\partial \Omega} \phi \ dx = 0$ is a normalization condition serving to fix the constant in the electrostatic potential determined by
the Poisson equation with Neumann boundary conditions (see the properties of the Neumann function in subsection \ref{sec:Neumann}). 
System (\ref{BDD}) is provided with non-negative initial data $(\bar{\rho}_0, \bar{n}_0)$ that satisfy 
\begin{equation} \label{initialdcm2}
 \int_\Omega \bar{\rho}_0 \ dx = \int_\Omega \bar{n}_0 \ dx = {\bar M} < +\infty.
\end{equation}
From the structure of the first two equations of system (\ref{BDD}) one observes that condition (\ref{initialdcm2}) is formally preserved for all $t \in [0,T[.$ \par
Drift-diffusion equations are commonly used to model semiconductor devices \cite{markowich}, and existence theories under various boundary conditions applicable to the semiconductor setting can be found in
\cite{GajewskiKroeger, jungel3}.
Drift-diffusion equations incorporate a gradient flow structure induced by the Wasserstein distance \cite{otto},  and its role in the 
bipolar drift-diffusion model is studied in \cite{monsaingeon}. In the one-dimensional case, \cite{DD1996} presents  a theory of weak solutions 
to a bipolar drift-diffusion system as the limit of a scaled sequence of entropy weak solutions of a bipolar Euler-Poisson system, while \cite{DDfrancesco} 
studies the long-time asymptotics.

\subsection{Bipolar Euler-Poisson}\label{sec:bipolar}
The bipolar Euler-Poisson system (\ref{BEP1}) describes the motion of two-species charged isentropic fluids subjected to an electric field.
 The system is formed by a pair of continuity equations for the densities, two momentum equations with friction, and a coupling Poisson equation for the electric field. The friction term is responsible for a damping force that gives rise to energy dissipation. \par
For a theory of existence of weak solutions to a bipolar Euler-Poisson system in the one-dimensional case, refer to \cite{DD1996}. There, compensated compactness is used to establish the existence of weak entropy solutions as the limit of a numerical approximation based on a modified fractional step Lax-Friedrich scheme. There it is also proved that these solutions satisfy $L^\infty$ and $L^2$ bounds.
\par
Regarding the structure of the momenum equations of system (\ref{BEP1}), observe that the frictional coefficient $1/ \varepsilon$ also multiplies the internal energy and electric field terms. From the bipolar Boltzmann-Poisson model with a Lenard-Bernstein collision operator \cite{transportjungel,lenard,markowich}, one formally derives the following system (see appendix for details) 
\begin{equation} \label{EP2}
 \begin{dcases} 
  \rho_t + \nabla \cdot (\rho u) = 0 \\ 
 (\rho u)_t  + \nabla \cdot (\rho u \otimes u) + \nabla p_1(\rho) = - \rho \nabla \phi -\frac{1}{\tau} \rho u  \\ 
 n_t + \nabla \cdot (n v) = 0  \\ 
 (n v)_t + \nabla \cdot (n v \otimes v)+\nabla p_2(n) = n \nabla \phi - \frac{1}{\tau} nv  \\ 
 -\Delta \phi = \rho - n,
 \end{dcases}
\end{equation} 
where $\rho$ and $n$ represent the densities of the fluids, $\rho u$ and $nv$ represent the momentums, $\phi$ stands for the electrostatic potential, and $\tau$ is the collision time. 
The formal limit of this system as $\tau \to 0$ is trivial in the hyperbolic scale. 
In the assumed Eulerian frame of reference, the collision time is usually much smaller than the observational time. 
For this reason one considers the time scaling $\partial / \partial t \to \tau \partial / \partial t$ (the so-called diffusion scaling), so that the observational time is measured in multiple units of the collision time. A change of scale is also applied for the velocities $u^\prime = \dfrac{ u}{\tau},$ $v^\prime = \dfrac{v}{\tau}$. After dropping the primes ($u^\prime \to u$, $v^\prime \to v$) and setting $\varepsilon = \tau^2$ one obtains the system 

\begin{equation} \label{BEP}
 \begin{dcases} 
   \rho_t + \nabla \cdot (\rho u) = 0 \\ 
 (\rho u)_t  + \nabla \cdot (\rho u \otimes u) + \frac{1}{\varepsilon} \nabla p_1(\rho) = - \frac{1}{\varepsilon} \rho \nabla \phi -\frac{1}{\varepsilon} \rho u  \\ 
 n_t + \nabla \cdot (n v) = 0  \\ 
 (n v)_t + \nabla \cdot (n v \otimes v)+\frac{1}{\varepsilon} \nabla p_2(n) = \frac{1}{\varepsilon} n \nabla \phi - \frac{1}{\varepsilon} nv  \\ 
 -\Delta \phi = \rho - n,
 \end{dcases}
\end{equation} 
\\
which is system (\ref{BEP1}) thanks to (\ref{bipderiv}) and (\ref{thermoconsistency}).
The collision time squared $\tau^2=\varepsilon$ is called the momentum relaxation time of system (\ref{BEP}). The limit of system (\ref{BEP}) as $\varepsilon \to 0$ is called the relaxation, overdamped or high-friction limit. 

System (\ref{BEP}) is prescribed with no-flux boundary conditions for the velocities and for the electric field. Precisely, the boundary conditions are
\begin{equation} \label{boundarycondBEP}
 u \cdot \nu = v \cdot \nu =  \frac{\partial \phi}{\partial \nu} = 0 \ \; \; \text{on} \ [0,T[ \times \partial \Omega, \ \ \ \int_{\partial \Omega} \phi \ dx = 0,
\end{equation}
where $\nu$ is an outer normal vector to $\partial \Omega$. Under this setting the fluids do not exit the set $\Omega$ and the system is electrically insulated.

For the initial datum $(\rho_0, u_0, n_0, v_0)$, we assume that  $\rho_0,n_0$ are non-negative and satisfy
\begin{equation} \label{initialdcm}
 \int_\Omega \rho_0 \ dx = \int_\Omega n_0 \ dx = M < +\infty,
\end{equation}
while  $u_0,v_0$ satisfy the no-flux condition at the boundary. 
The continuity equations together with \eqref{boundarycondBEP} imply that 
condition (\ref{initialdcm}) is  preserved for all times $t \in [0,T[.$

\subsection{Relative energy identity for the bipolar fluid system} \label{sec3}

In this section, a relative energy identity for solutions of (\ref{BEP}) is derived. This identity, expression \eqref{REBEPEvol}, is produced by an exact formal calculation using the abstract formalism presented in \cite{GLT}. Then, in Proposition \ref{relativeentropy}, we outline an argument producing the relative energy calculation between a weak dissipative solution of (\ref{BEP}) and a strong and bounded away from vacuum solution of (\ref{BDD}). \par
The potential energy functional that generates the bipolar fluid system is
$$
\begin{aligned}
\mathcal{E}=\mathcal{E}(\rho,n) &= \int_\Omega h_1(\rho) + h_2(n) + \tfrac{1}{2}|\nabla \phi|^2 dx
\\
&= \int_\Omega h_1(\rho) + h_2(n) +  \tfrac{1}{2} (\rho -n)  \big ( N \ast (\rho - n) \big ) dx.
\end{aligned}
$$ 
Using \eqref{thermoconsistency} one can readily see
\begin{equation} \label{rhopstress}
 -\rho \nabla \frac{\delta \mathcal{E}}{\delta \rho}(\rho,n) = \nabla \cdot S_1(\rho) - \rho \nabla \big(N * (\rho - n) \big),
\end{equation}
\begin{equation} \label{npstress}
-n \nabla \frac{\delta \mathcal{E}}{\delta n}(\rho,n) = \nabla \cdot S_2(n) + n \nabla \big(N * (\rho - n) \big),
\end{equation} \\
where $S_1(\rho) = -p_1(\rho)I,$ $S_2(n) = -p_2(n)I$ are the pressure stresses. \par
To develop the relative energy identity for (\ref{BEP}), some preliminary formulas are needed. 
Let $\varphi$ be a vector valued test function, and write the weak forms of the identities (\ref{rhopstress}) and (\ref{npstress})
\begin{equation} \label{weakrhopstress}
 \big< \frac{\delta \mathcal{E}}{\delta \rho} (\rho,n), \nabla \cdot (\rho \varphi)  \big> = - \int_\Omega S_1(\rho) : \nabla \varphi \ dx + \int_\Omega \big(N*(\rho - n) \big) \nabla \cdot (\rho \varphi)dx,
\end{equation}
\begin{equation} \label{weaknpstress}
\big< \frac{\delta \mathcal{E}}{\delta n} (\rho,n), \nabla \cdot (n \varphi)  \big> = - \int_\Omega S_2(n) : \nabla \varphi \ dx - \int_\Omega \big(N*(\rho - n) \big) \nabla \cdot (n \varphi)dx.
\end{equation} \\
Taking the directional derivative of \eqref{weakrhopstress} in the direction $\rho$ and of \eqref{weaknpstress} in the direction $n$ we respectively obtain
\begin{equation} \label{2ndformularho}
\begin{aligned}
 \big< \big<  \frac{\delta^2 \mathcal{E}}{\delta \rho^2}(\rho,n), \ & \big(\nabla \cdot (\rho \varphi), \psi \big)  \big> \big> + \big< \frac{\delta \mathcal{E}}{\delta \rho} (\rho,n), \nabla \cdot (\psi \varphi)  \big> = \\
 =& - \int_\Omega \big< \frac{\delta S_1}{\delta \rho}(\rho), \psi \big> : \nabla \varphi \ dx + \int_\Omega (N*\psi) \nabla \cdot (\rho \varphi) dx \\
 &+\int_\Omega \big(N*(\rho - n) \big) \nabla \cdot (\psi \varphi)dx,
\end{aligned}
\end{equation}

\begin{equation} \label{2ndformulan}
\begin{aligned}
 \big< \big<  \frac{\delta^2 \mathcal{E}}{\delta n^2}(\rho,n), \ & \big(\nabla \cdot (n \varphi), \psi \big)  \big> \big> + \big< \frac{\delta \mathcal{E}}{\delta n} (\rho,n), \nabla \cdot (\psi \varphi)  \big> = \\
 =& - \int_\Omega \big< \frac{\delta S_2}{\delta n}(n), \psi \big> : \nabla \varphi \ dx + \int_\Omega (N*\psi) \nabla \cdot (n \varphi) dx \\
 &-\int_\Omega \big(N*(\rho - n) \big) \nabla \cdot (\psi \varphi)dx,
\end{aligned}
\end{equation} 
where $\psi$ is a scalar test function.
Finally, we need a formula for the mixed functional derivatives of  $\mathcal{E}$, which, for scalar valued test functions $\alpha, \beta$, takes the form:
\begin{equation} \label{crossfunctderiv}
 \big< \big<  \frac{\delta^2 \mathcal{E}}{\delta \rho \delta n}(\rho,n), \ ( \alpha, \beta ) \big> \big> =  \big< \big<  \frac{\delta^2 \mathcal{E}}{\delta n \delta \rho}(\rho,n), \ ( \alpha, \beta ) \big> \big> = - \int_\Omega (N* \beta) \alpha \ dx.
\end{equation} \\
Next, define the relative functional associated with $\mathcal{E}$,  given by the quadratic part of the Taylor series expansion of $\mathcal{E}$.
Given density pairs $(\rho, n)$, $(\bar \rho, \bar n)$, with $\phi= N \ast (\rho -n)$, $\bar \phi = N \ast (\bar \rho - \bar n),$ the relative potential energy functional
 is defined by 
 $$\mathcal{E}(\rho, n | \bar{\rho}, \bar{n}) \coloneqq \mathcal{E}(\rho,n)- \mathcal{E}(\bar{\rho}, \bar{n}) - \big<\frac{\delta \mathcal{E}}{\delta \rho}(\bar{\rho}, \bar{n}) ,\rho - \bar{\rho} \big> - \big<\frac{\delta \mathcal{E}}{\delta n}(\bar{\rho}, \bar{n}) ,n - \bar{n} \big>.
 $$
 A straightforward computation gives
$$\mathcal{E}(\rho, n | \bar{\rho}, \bar{n}) = \int_\Omega h_1(\rho | \bar{\rho}) + h_2(n | \bar{n}) + \tfrac{1}{2} |\nabla (\phi-\bar{\phi})|^2dx,$$
where $h(r| \bar{r}) \coloneqq h(r)-h(\bar{r}) - h^\prime(\bar{r})(r-\bar{r}).$ \par

Now let $(\rho, \rho u, n, nv), \ (\bar \rho, \bar \rho \bar u, \bar n, \bar n \bar v),$ with $\phi = N * (\rho - n), \ \bar \phi = N * (\bar \rho - \bar n),$ be two smooth solutions of (\ref{BEP}). Using the continuity equations one derives 

\begin{equation} \label{formalRE1}
\begin{aligned}
\frac{d}{dt}\mathcal{E}&(\rho , n | \bar \rho, \bar n) = \frac{d}{dt} \Big( \mathcal{E}(\rho,n)-\mathcal{E}(\bar \rho, \bar n) - \big< \frac{\delta \mathcal{E}}{\delta \rho}(\bar \rho, \bar n), \rho - \bar \rho \big>  - \big< \frac{\delta \mathcal{E}}{\delta n}(\bar \rho, \bar n), n - \bar n \big> \Big)
\\
= &\left ( - \big<\frac{\delta \mathcal{E}}{\delta \rho}(\rho,n), \nabla \cdot (\rho u) \big> + \big<\frac{\delta \mathcal{E}}{\delta \rho}(\bar \rho,\bar n), \nabla \cdot (\bar \rho \bar u) \big> + \big<\frac{\delta \mathcal{E}}{\delta \rho}(\rho,n), \nabla \cdot \big(\rho (u-\bar u) \big) \big> \right . \\
& + \big< \big<  \frac{\delta^2 \mathcal{E}}{\delta \rho^2}(\bar \rho,\bar n), \ \big(\nabla \cdot (\bar \rho \bar u), \rho - \bar \rho \big)  \big> \big> + \big< \frac{\delta \mathcal{E}}{\delta \rho} (\bar \rho,\bar n), \nabla \cdot \big((\rho -\bar \rho) \bar u \big)  \big> \\
& \left . + \big< \big<  \frac{\delta^2 \mathcal{E}}{\delta \rho \delta n}(\bar \rho,\bar n), \ \big( \nabla \cdot (\bar \rho \bar u), n - \bar n \big) \big> \big>  \right ) \\
&  - \big< \frac{\delta \mathcal{E}}{\delta \rho}(\rho,n) - \frac{\delta \mathcal{E}}{\delta \rho}(\bar \rho, \bar n),  \nabla \cdot \big(\rho(u-\bar u) \big) \big>   \\
&+ \left ( - \big<\frac{\delta \mathcal{E}}{\delta n}(\rho,n), \nabla \cdot (n v) \big> + \big<\frac{\delta \mathcal{E}}{\delta \rho}(\bar \rho,\bar n), \nabla \cdot (\bar n \bar v) \big> + \big<\frac{\delta \mathcal{E}}{\delta \rho}(\rho,n), \nabla \cdot \big(n (v-\bar v) \big) \big> \right .  \\
& + \big< \big<  \frac{\delta^2 \mathcal{E}}{\delta n^2}(\bar \rho,\bar n), \ \big(\nabla \cdot (\bar n \bar v), n - \bar n \big)  \big> \big> + \big< \frac{\delta \mathcal{E}}{\delta n} (\bar \rho,\bar n), \nabla \cdot \big((n -\bar n) \bar v \big)  \big> \\
&\left .  + \big< \big<  \frac{\delta^2 \mathcal{E}}{\delta n \delta \rho}(\bar \rho,\bar n), \ \big( \nabla \cdot (\bar n \bar v), \rho - \bar \rho \big) \big> \big>  \right ) \\
& - \big< \frac{\delta \mathcal{E}}{\delta n}(\rho,n) - \frac{\delta \mathcal{E}}{\delta n}(\bar \rho, \bar n),  \nabla \cdot \big(n(v-\bar v) \big) \big> \\
=:& \ I_{\rho u}   - \big< \frac{\delta \mathcal{E}}{\delta \rho}(\rho,n) - \frac{\delta \mathcal{E}}{\delta \rho}(\bar \rho, \bar n),  \nabla \cdot \big(\rho(u-\bar u) \big) \big> \\
&+ I_{n v}  - \big< \frac{\delta \mathcal{E}}{\delta n}(\rho,n) - \frac{\delta \mathcal{E}}{\delta n}(\bar \rho, \bar n),  \nabla \cdot \big(n(v-\bar v) \big) \big>
\end{aligned}
\end{equation}\\
where $I_{\rho u}$ and $I_{n v}$ are the terms in parentheses, respectively.

To compute $I_{\rho u},$ one applies formula (\ref{weakrhopstress}) to its first three terms, formula (\ref{2ndformularho}) with $\psi = \rho - \bar \rho,$ $\varphi = \bar u$ to its fourth and fifth terms, and formula (\ref{crossfunctderiv}) with $\alpha = \nabla \cdot (\bar \rho \bar u),$ $\beta = n - \bar n$ to its last term to obtain
\begin{equation} \label{Irhou}
\begin{aligned}
I_{\rho u}  = &  \int_\Omega \big( S_1(\rho) - S_1(\bar \rho) - \big< \frac{\delta S_1}{\delta \rho}(\bar \rho), \rho - \bar \rho\big> \big) : \nabla \bar u \ dx \\
& - \int_\Omega \big(N*(\rho - n - \bar \rho + \bar n) \big) \nabla \cdot \big( (\rho - \bar \rho) \bar u \big)dx \\
 = & \int_\Omega S_1(\rho | \bar \rho) : \nabla \bar u \ dx + \int_\Omega (\rho - \bar \rho) \bar u \cdot \nabla (\phi - \bar \phi)dx.
\end{aligned}
\end{equation}
In a similar fashion, using formulas (\ref{weaknpstress}), (\ref{2ndformulan}), (\ref{crossfunctderiv}) one derives 
\begin{equation} \label{Inv}
I_{nv}  = \int_\Omega S_2(n | \bar n) : \nabla \bar v \ dx - \int_\Omega (n - \bar n) \bar v \cdot \nabla (\phi - \bar \phi)dx.
\end{equation}
Substituting (\ref{Irhou}) and (\ref{Inv}) into (\ref{formalRE1}) it yields
\begin{equation} \label{formalRE2}
 \begin{aligned}
  \frac{d}{dt} \mathcal{E}(\rho,n | \bar \rho, \bar n)  = & \int_\Omega S_1(\rho | \bar \rho) : \nabla \bar u + S_2(n | \bar n) : \nabla \bar v\ dx + \int_\Omega \big( (\rho - \bar \rho) \bar u - (n - \bar n) \bar v \big)\cdot \nabla (\phi - \bar \phi)dx \\
  & - \big< \frac{\delta \mathcal{E}}{\delta \rho}(\rho,n) - \frac{\delta \mathcal{E}}{\delta \rho}(\bar \rho, \bar n),  \nabla \cdot \big(\rho(u-\bar u) \big) \big> - \big< \frac{\delta \mathcal{E}}{\delta n}(\rho,n) - \frac{\delta \mathcal{E}}{\delta n}(\bar \rho, \bar n),  \nabla \cdot \big(n(v-\bar v) \big) \big>.
 \end{aligned}
 \end{equation} \par 
 
To reach identity (\ref{formalRE2}) one has only used the continuity equations and the structure of the potential energy functional $\mathcal{E}.$ In order to exploit the structure of the momentum equations, we consider the kinetic energy functional $\mathcal{K}$ for the bipolar fluid system, 
 $$\mathcal{K}(\rho, \rho u, n, nv) = \int_\Omega \tfrac{1}{2}\rho |u|^2 + \tfrac{1}{2}n |v|^2dx,
 $$
 and compute the relative kinetic energy 
 \begin{equation*}
 \begin{split}
  \mathcal{K}(\rho, \rho u, n, nv | \bar{\rho}, \bar{\rho} \bar{u}, \bar{n}, \bar{n} \bar{v}) = & \ \mathcal{K}(\rho, \rho u, n, nv) - \mathcal{K}(\bar{\rho}, \bar{\rho} \bar{u}, \bar{n}, \bar{n} \bar{v}) \\ 
  & - \big< \frac{\delta \mathcal{K}}{\delta \rho}(\bar{\rho}, \bar{\rho} \bar{u}, \bar{n}, \bar{n} \bar{v}) , \rho - \bar{\rho}\big> - \big< \frac{\delta \mathcal{K}}{\delta (\rho u)}(\bar{\rho}, \bar{\rho} \bar{u}, \bar{n}, \bar{n} \bar{v}) , \rho u - \bar{\rho u}\big> \\
  & - \big< \frac{\delta \mathcal{K}}{\delta n}(\bar{\rho}, \bar{\rho} \bar{u}, \bar{n}, \bar{n} \bar{v}) , n - \bar{n}\big> - \big< \frac{\delta \mathcal{K}}{\delta (n v)}(\bar{\rho}, \bar{\rho} \bar{u}, \bar{n}, \bar{n} \bar{v}) , n v - \bar{n v}\big>.
  \\
  & = \int_\Omega \tfrac{1}{2} \rho |u - \bar{u}|^2 + \tfrac{1}{2} n |v - \bar{v}|^2 dx,
 \end{split}
 \end{equation*}
 where $(\rho, \rho u, n, nv), \ (\bar \rho, \bar \rho \bar u, \bar n, \bar n \bar v)$ are two pairs of densities and momenta. \par
 Let $(\rho, \rho u, n, nv), \ (\bar \rho, \bar \rho \bar u, \bar n, \bar n \bar v),$ with $\phi = N * (\rho - n), \ \bar \phi = N * (\bar \rho - \bar n),$ be two solutions of (\ref{BEP}). 
 Consider the momentum equation satisfied by the difference $u-\bar u,$
 $$\varepsilon (u - \bar u)_t + \varepsilon (u \cdot \nabla)(u- \bar u) + \varepsilon \big( (u - \bar u) \cdot \nabla \big) \bar u = - \nabla \Big( \frac{\delta \mathcal{E}}{\delta \rho}(\rho,n) - \frac{\delta \mathcal{E} }{\delta \rho}(\bar \rho, \bar n) \Big)-(u-\bar u).$$
Taking the inner product with $\rho (u - \bar u)$ and using the continuity equation gives
 $$
  \begin{aligned}
   \big(\varepsilon \tfrac{1}{2} \rho |u - \bar u|^2 \big)_t + \varepsilon \nabla \cdot \big( \tfrac{1}{2} \rho |u - & \bar u|^2 u \big) + \varepsilon  \nabla \bar u : \rho (u - \bar u) \otimes (u - \bar u) =  \\
   & = - \nabla \Big( \frac{\delta \mathcal{E}}{\delta \rho}(\rho,n) - \frac{\delta \mathcal{E} }{\delta \rho}(\bar \rho, \bar n) \Big)\cdot \big(\rho (u - \bar u) \big)-\rho|u-\bar u|^2.
  \end{aligned}
$$
Analogously,
$$
  \begin{aligned}
   \big(\varepsilon \tfrac{1}{2} n |v - \bar v|^2 \big)_t + \varepsilon \nabla \cdot \big( \tfrac{1}{2} n |v - & \bar v|^2 v \big) + \varepsilon  \nabla \bar v : n (v - \bar v) \otimes (v - \bar v) = \\
   & = - \nabla \Big( \frac{\delta \mathcal{E}}{\delta n}(\rho,n) - \frac{\delta \mathcal{E} }{\delta n}(\bar \rho, \bar n) \Big)\cdot \big(n (v - \bar v) \big)-n|v-\bar v|^2.
  \end{aligned}
$$
Adding the  previous expressions and integrating over space renders the evolution of the relative kinetic energy
 
 \begin{equation} \label{evolrelativekinetic}
  \begin{aligned}
   \frac{d}{dt} \int_\Omega & \varepsilon \tfrac{1}{2} \rho |u - \bar{u}|^2 + \varepsilon \tfrac{1}{2} n |v - \bar{v}|^2 dx + \int_\Omega \rho|u-\bar u|^2 + n |v - \bar v|^2dx = \\
   = & - \varepsilon \int_\Omega \nabla \bar{u} : \rho (u-\bar{u}) \otimes (u-\bar{u})+ \nabla \bar{v} : n (v-\bar{v}) \otimes  (v-\bar{v})dx \\
   & + \big< \frac{\delta \mathcal{E}}{\delta \rho}(\rho,n) - \frac{\delta \mathcal{E}}{\delta \rho}(\bar \rho, \bar n),  \nabla \cdot \big(\rho(u-\bar u) \big) \big> + \big< \frac{\delta \mathcal{E}}{\delta n}(\rho,n) - \frac{\delta \mathcal{E}}{\delta n}(\bar \rho, \bar n),  \nabla \cdot \big(n(v-\bar v) \big) \big>.
  \end{aligned}
\end{equation}

\smallskip
The relative energy identity is then obtained by adding (\ref{formalRE2}) with (\ref{evolrelativekinetic}):
\begin{equation} \label{REBEPEvol}
\begin{split}
 \frac{d}{dt}  \Big( & \mathcal{E}(\rho, n | \bar{\rho}, \bar{n}) + \varepsilon \int_\Omega \tfrac{1}{2} \rho |u - \bar u|^2 + \tfrac{1}{2} n |v - \bar v|^2  dx\Big) 
  +  \int_\Omega \rho |u - \bar{u}|^2 + n |v - \bar{v}|^2 dx =
\\
  = &  - \varepsilon \int_\Omega \nabla \bar{u} : \rho (u-\bar{u}) \otimes (u-\bar{u})+ \nabla \bar{v} : n (v-\bar{v}) \otimes  (v-\bar{v})dx \\
  & +  \int_\Omega S_1(\rho | \bar{\rho}) : \nabla \bar{u} + S_2(n | \bar{n}) : \nabla \bar{v} \ dx +  \int_\Omega \big( (\rho -\bar{\rho}) \bar{u} - (n - \bar{n})\bar{v} \big) \cdot \nabla(\phi- \bar{\phi}) dx.
\end{split}
\end{equation} \par

\section{Statement of the main result}
\label{sec4}

The main objective of this work is to compare a dissipative weak solution $(\rho,\rho u,n,nv)$, with $\phi = N \ast (\rho - n)$, of the bipolar Euler-Poisson system \eqref{BEP}
with a strong and bounded away from vacuum solution $(\bar \rho, \bar n)$, $\bar \phi = N \ast (\bar \rho - \bar n)$, of the bipolar drift-diffusion system \eqref{BDD}.
As precised in subsection \ref{sec:bipolar}, the limit $\varepsilon \to 0$  corresponds to the overdamping limit for the bipolar Euler-Poisson system.  
The plan of the section is to first precise
the notions of solutions utilized in this work, then describe the methodology of comparison via the relative energy, and finally state the main result.

The internal energy functions
$h_1,h_2 $ and the pressure functions $p_1,p_2$ are assumed to satisfy, apart from
(\ref{thermoconsistency}), the limiting behaviors
\begin{equation} \label{ho(g)}
\lim\limits_{r \to +\infty} \frac{h_i(r)}{r^{\gamma_i}} = \frac{k_i}{\gamma_i -1}, 
\end{equation}
and 
\begin{equation}\label{pbound}
|p_i^{\prime \prime}(r)| \leq \hat{k}_i \frac{p_i^\prime(r)}{r}, \ r>0,
\end{equation}
for some exponents $\gamma_1, \gamma_2 > 1$ and  
for some positive constants $k_i,\hat{k}_i, \ i=1,2$.
The prototypical examples satisfying these conditions are
$$p(r) = k r^\gamma, \ \ \ h(r)= \frac{k}{\gamma -1}r^\gamma,$$ where $\gamma > 1$ and $k > 0$.

\medskip

First we describe the assumptions on the solution of the bipolar Euler-Poisson system \eqref{BEP}.
\begin{definition} \label{weakformulation}
 The vector function $(\rho,\rho u,n,nv)$ with $\rho, n \ge 0$ and regularity
 $$\rho \in C\big([0,T[;  L^{\gamma_1}(\Omega)\big), \qquad  \ n \in C\big([0,T[; L^{\gamma_2}(\Omega)\big),$$ 
 $$\rho u, \ nv \in C\Big([0,T[;\big(L^1(\Omega)\big)^d\Big),$$
 $$\rho |u|^2, \ n |v|^2 \in  L^1 \big ( ]0, T[ \times \Omega) \big), $$
together with $\phi = N * (\rho -n)$ satisfying
$$
\rho \nabla \phi , \; n \nabla \phi \in L^1 \left ( ]0,T[ \times \Omega \right ) 
$$
is a weak solution of (\ref{BEP}) provided that:
\begin{enumerate}[(i)]
 \item $(\rho,\rho u,n,nv)$ satisfies (\ref{BEP}) in the weak sense
 \begin{equation} \label{weak1}
         -\int_0^T \int_\Omega \varphi_t \rho \ dxdt-\int_0^T \int_\Omega \nabla \varphi \cdot (\rho u) dxdt - \int_\Omega \varphi \rho \big|_{t=0}dx=0,
        \end{equation}
 \begin{equation} \label{weak2}
        \begin{split}
        -\varepsilon \int_0^T \int_\Omega \tilde{\varphi}_t \cdot (\rho u)dxdt - \varepsilon \int_0^T \int_\Omega \nabla \tilde{\varphi} : \rho u \otimes u \ dx dt -  \int_0^T \int_\Omega (\nabla \cdot \tilde{\varphi}) p_1(\rho)  dxdt \\ 
        -\varepsilon \int_\Omega \tilde{\varphi} \cdot (\rho u)\big|_{t=0}dx =  -  \int_0^T \int_\Omega \tilde{\varphi}\cdot (\rho \nabla \phi)dxdt -  \int_0^T \int_\Omega \tilde{\varphi} \cdot (\rho u)dxdt,
        \end{split}
        \end{equation}
\begin{equation} \label{weak3}
         -\int_0^T \int_\Omega \psi_t n \ dxdt-\int_0^T \int_\Omega \nabla \psi \cdot (nv) dxdt - \int_\Omega \psi n \big|_{t=0}dx=0,
        \end{equation}
 \begin{equation} \label{weak4}
        \begin{split}
         -\varepsilon \int_0^T \int_\Omega \tilde{\psi}_t \cdot (nv)dxdt - \varepsilon \int_0^T \int_\Omega \nabla \tilde{\psi} :  nv \otimes v \ dx dt -  \int_0^T \int_\Omega (\nabla \cdot \tilde{\psi}) p_2(n)  dxdt \\ 
         - \varepsilon \int_\Omega \tilde{\psi} \cdot (nv)\big|_{t=0}dx =   \int_0^T \int_\Omega \tilde{\psi} \cdot (n \nabla \phi)dxdt - \int_0^T \int_\Omega \tilde{\psi} \cdot (nv)dxdt,
        \end{split}
        \end{equation}
      for all Lipschitz test functions $\varphi, \psi : [0,T[ \times \bar{\Omega} \to \mathbb{R}, \ \tilde{\varphi}, \tilde{\psi}:[0,T[ \times \bar{\Omega} \to \mathbb{R}^d$ compactly supported in time and satisfying $\tilde{\varphi} \cdot \nu = \tilde{\psi} \cdot \nu = 0 $ on $[0,T[ \times \partial \Omega,$ where $\nu$ is any outer normal vector to the boundary;
    \item $(\rho,\rho u,n,nv)$ is equipped with the bounds:
    \begin{equation} \label{massconservation}
 \int_\Omega \rho \ dx = \int_\Omega n \ dx = M < +\infty, \ \ \ \forall t \in [0,T[,
  \end{equation}
  \begin{equation} \label{energyconservation}
  \underset{[0,T[}{\text{sup}} \int_\Omega  \varepsilon\tfrac{1}{2}\rho|u|^2+\varepsilon \tfrac{1}{2}n|v|^2+h_1(\rho)+h_2(n)+ \tfrac{1}{2} |\nabla \phi|^2  dx < +\infty;
 \end{equation}
\end{enumerate}
\end{definition}

\medskip
\noindent
Solutions of (\ref{BEP}) clearly depend on $\varepsilon$, that is  $(\rho,\rho u,n,nv)=(\rho_\varepsilon,\rho_\varepsilon u_\varepsilon,n_\varepsilon,n_\varepsilon v_\varepsilon)$; this dependence is supressed for simplicity. 
 Property (\ref{massconservation}) represents the conservation of mass for $\rho$ and $n$, whereas (\ref{energyconservation}) asserts that
 the total energy is finite.

\begin{definition}
 A weak solution $(\rho,\rho u,n,nv), \ \phi = N*(\rho -n),$ of (\ref{BEP}) is called dissipative if $\rho |u|^2, \ n |v|^2, \ |\nabla \phi|^2 \in C\big([0,T[; L^1(\Omega) \big),$ and it satisfies 
 \begin{equation} \label{weakdissip}
   \begin{split}
-  \int_0^T  & \int_\Omega \big(\varepsilon \tfrac{1}{2}\rho|u|^2+\varepsilon \tfrac{1}{2}n|v|^2+ h_1(\rho)+h_2(n)+ \tfrac{1}{2} |\nabla \phi|^2\big) \dot{\theta}(t)dxdt \\
    + & \int_0^T \int_\Omega\big( \rho |u|^2+ n |v|^2\big) \theta(t)dxdt \\
    \leq   \int_\Omega & \big(\varepsilon \tfrac{1}{2}\rho|u|^2+\varepsilon \tfrac{1}{2}n|v|^2+h_1(\rho)+h_2(n)+ \tfrac{1}{2} |\nabla \phi|^2\big)\big|_{t=0} \theta(0) dx
   \end{split}
  \end{equation} \\ 
  for any non-negative $\theta \in W^{1,\infty}([0,T[)$ with compact support. 
\end{definition} \par

 Next, we turn to solutions $(\bar{\rho}, \bar{n}),$ with $\bar{\phi}= N \ast (\bar{\rho}- \bar{n}),$ of the bipolar drift-diffusion system.
These are assumed to be classical solutions of \eqref{BDD} which satisfy the boundary conditions (\ref{boundarycondBDD}), 
and emanate from initial data satisfying the bounds 
  \begin{align}
   \int_\Omega \bar \rho_0 \ dx = \int_\Omega \bar n_0 \ dx = \bar M < +\infty,
  \\
  \int_\Omega h_1(\bar{\rho}_0)+h_2(\bar{n}_0)+ \tfrac{1}{2} |\nabla \bar{\phi}_0|^2  dx < +\infty,
   \label{energyconservationDD}
  \end{align}
where $\bar \phi_0 = N*(\bar \rho_0 - \bar n_0).$ Setting 
\begin{equation}\label{defappsol}
 \bar{u} \coloneqq -\nabla\big(h_1^{\prime}(\bar{\rho})+\bar{\phi}\big), \ \ \ \bar{v} \coloneqq -\nabla\big(h_2^{\prime}(\bar{n})-\bar{\phi}\big),
 \end{equation}
after multiplying the previous expressions by $\bar \rho \bar u$ and $\bar n \bar v,$ respectively, and integrating over space one obtains the energy identity for system (\ref{BDD}):
\begin{equation}\label{energyBDD}
 \frac{d}{dt} \int_\Omega h_1(\bar{\rho})+h_2(\bar{n})+ \tfrac{1}{2} |\nabla \bar{\phi}|^2 dx = - \int_\Omega \bar \rho | \bar u |^2 + \bar n |\bar v|^2dx.
\end{equation}
Due to (\ref{energyBDD}), condition (\ref{energyconservationDD}) is preserved for all times $t \in [0,T[.$
Expressing the energy identity (\ref{energyBDD}) in a weak form, one has that a strong solution of (\ref{BDD}) satisfies
\begin{equation} \label{strongdissip}
   \begin{split}
& -   \int_0^T   \int_\Omega \big(h_1(\bar{\rho})+h_2(\bar{n})+ \tfrac{1}{2} |\nabla \bar{\phi}|^2\big) \dot{\theta}(t)dxdt +  \int_0^T \int_\Omega\big( \bar{\rho} |\bar{u}|^2+ \bar{n} |\bar{v}|^2\big) \theta(t)dxdt \\
     & =   \int_\Omega  \big(h_1(\bar{\rho})+h_2(\bar{n})+ \tfrac{1}{2} |\nabla \bar{\phi}|^2\big)\big|_{t=0} \theta(0) dx,
   \end{split}
  \end{equation} 
 for all $\theta \in W^{1,\infty}([0,T[)$ with compact support.
 \par
Moreover, $(\bar{\rho}, \bar{n})$ is assumed to be bounded away from vacuum: \\ 
\textbf{(H)}  There exist $\delta_1, \delta_2 > 0$ and $M_1,M_2 < +\infty$ such that 
$$
\bar{\rho}(t,x) \in [\delta_1, M_1] \, , \quad \bar{n}(t,x) \in [\delta_2, M_2] \quad \mbox{ for $(t,x) \in [0,T[\times \Omega$. }
$$
\par

In order to compare $(\rho, \rho u, n, n v, \phi)$ with $(\bar \rho, \bar n, \bar \phi)$ we proceed along the lines of \cite{lattanziothanos2013} and view 
$(\bar \rho, \bar n, \bar \phi)$ as an approximate solution of \eqref{BEP}. This is accomplished by setting $(\bar u, \bar v)$ via \eqref{defappsol}.
We refer to the resulting $(\bar \rho, \bar \rho u, \bar n , \bar n \bar v)$ as a strong and bounded away from vacuum solution of \eqref{BDD}. The regularity ``strong'' refers to the boundedness of all of its derivatives that will appear later. Precisely, one requires that the derivatives 
$$ 
 \dfrac{\partial  \bar{\rho}}{\partial t},  \
  \dfrac{\partial \bar{n}}{\partial t},  \
 \dfrac{\partial ^2 \bar{\rho}}{\partial x_i\partial t}, \
 \dfrac{\partial ^2 \bar{n}}{\partial x_i \partial t}, \
 \dfrac{\partial ^2 \bar{\phi}}{\partial x_i\partial t},\
 \dfrac{\partial^2 \bar{\rho}}{\partial x_i \partial x_j}, \
 \dfrac{\partial^2 \bar{n}}{\partial x_i \partial x_j}, \
 \dfrac{\partial^2 \bar{\phi}}{\partial x_i \partial x_j}
$$
are in $L^\infty([0,T[ \times \Omega)$ for all $i,j = 1, \ldots, d.$ \par
One easily checks that
\begin{equation*} 
         \bar{\rho}_t + \nabla \cdot(\bar{\rho} \bar{u}) = 0, 
\end{equation*}
\begin{equation*} 
        \bar{n}_t+ \nabla \cdot(\bar{n} \bar{v}) = 0, 
\end{equation*}
and 
\begin{equation} \label{nofluxrhou}
 \bar \rho \bar{u} \cdot \nu = \bar n \bar{v} \cdot \nu = 0 \ \text{on} \ [0,T[ \times \partial \Omega
\end{equation}
for $\nu$ an outer normal vector to $\partial \Omega.$ 
Then setting 
$$\bar{e}_1 \coloneqq (\bar{\rho} \bar{u})_t + \nabla \cdot (\bar{\rho} \bar{u} \otimes \bar{u}),$$  $$\bar{e}_2 \coloneqq (\bar{n} \bar{v})_t + \nabla \cdot (\bar{n} \bar{v} \otimes \bar{v}),$$ the equilibrium system (\ref{BDD}) can be rewritten as an approximation of the system \eqref{BEP}, 
\begin{equation} \label{BDDlifted}
\begin{dcases} 
 \bar{\rho}_t + \nabla \cdot (\bar{\rho} \bar{u}) = 0 \\ 
 (\bar{\rho} \bar{u})_t + \nabla \cdot (\bar{\rho} \bar{u} \otimes \bar{u}) = -\frac{1}{\varepsilon} \bar{\rho} \nabla(h_1^{\prime}(\bar{\rho})+\bar{\phi})-\frac{1}{\varepsilon} \bar{\rho} \bar{u} + \bar{e}_1 \\ 
 \bar{n}_t + \nabla \cdot (\bar{n} \bar{v}) = 0  \\ 
 (\bar{n} \bar{v})_t + \nabla \cdot (\bar{n} \bar{v} \otimes \bar{v}) = -\frac{1}{\varepsilon} \bar{n} \nabla(h_2^{\prime}(\bar{n})-\bar{\phi})-\frac{1}{\varepsilon} \bar{n} \bar{v} + \bar{e}_2\\ 
 -\Delta \bar{\phi} = \bar{\rho} - \bar{n}.
 \end{dcases}
\end{equation} \par
where
$$
\bar{e}_1, \bar{e}_2 \in \big(L^\infty(]0,T[ \times \Omega)\big)^d.
$$ 

The two solutions $(\rho, \rho u, n, n v)$ and $(\bar \rho, \bar \rho u, \bar n , \bar n \bar v),$ with $\phi=N*(\rho -n), \ \bar \phi = N*(\bar \rho - \bar n)$ are then compared by means of  the relative energy  
$\Psi : [0,T[ \to \mathbb{R}$ for \eqref{BEP} given by
$$
\Psi(t)=\int_\Omega \varepsilon \tfrac{1}{2} \rho |u - \bar{u}|^2+\varepsilon \tfrac{1}{2} n |v - \bar{v}|^2 + h_1(\rho | \bar{\rho}) + h_2(n | \bar{n}) +\tfrac{1}{2} |\nabla (\phi - \bar{\phi})|^2dx. 
$$

We prove:
\begin{theorem}\label{mainresult}
Let $(\rho,\rho u,n,n v)$, with $\phi = N*(\rho - n)$,  be a dissipative weak solution of (\ref{BEP}) with $\gamma_1,\gamma_2 \geq 2 - \frac{1}{d},$ and
let $(\bar{\rho}, \bar \rho \bar u, \bar{n}, \bar n \bar v)$,  with $\bar{\phi}= N*(\bar{\rho}-\bar{n})$,  be a strong and bounded away from vacuum solution of \eqref{BDD}.
There exists $C > 0$ such that for $t \in [0, T [$ the relative energy $\Psi$ between these two solutions satisfies the stability estimate
$$ 
\Psi(t) \leq e^{CT}\big(\Psi(0) + \varepsilon^2 \big). 
$$
Therefore if $\Psi(0) \to 0$ as $\varepsilon \to 0$, then $\Psi(t) \to 0$ as $\varepsilon \to 0$ for every $t \in [0, T [$.
\end{theorem}

\section{Convergence in the relaxation limit}
\label{sec5}
This section contains the proof of Theorem \ref{mainresult}. We start with some auxilliary results on the behavior of the Neumann function and Riesz potentials,
then continue with the derivation of the relative energy identity within the regularity class detailed in section \ref{sec4}, and conclude with the proof 
of the stability estimate.

\subsection{Auxiliary results} \label{sec:Neumann}
Regarding the Neumann function $N \in C^\infty ( \bar{\Omega} \times \bar{\Omega} \setminus \{(x,x) \ | \ x \in \bar{\Omega} \} )$, the relevant properties that will be used are \cite[Chapter 1, Section 6]{carloskenig}:
\begin{enumerate}[(i)]
 \item $N(x,y) = N(y,x),$ 
 \item $N(x,y) \leq \dfrac{C}{|x-y|^{d-2}},$ 
 \item $\nabla_x N(x,y) \leq \dfrac{C}{|x-y|^{d-1}},$ 
 \item If $f \in H^1(\Omega)^* \cap W^{1,p}(\Omega)^*,$ for $p < d/(d-1),$ satisfies $\int_\Omega f \ dx = 0,$ then $\beta = N * f$ is the unique solution of 
 $$\int_\Omega \nabla \beta \cdot \nabla \varphi \ dx = \int_\Omega f \varphi \ dx \ \ \forall \varphi \in H^1(\Omega)$$ that satisfies $\int_{\partial \Omega} \beta \ dx = 0$ and belongs to $C^\alpha(\bar{\Omega}),$ with $\alpha$ depending only on $d.$
\end{enumerate}  
\vspace{3mm}
\par
In order to deal with the electrostatic potential $\phi$ one needs to recall the notion of Riesz potential. Given a function $f : \mathbb{R}^d \to \mathbb{R}$, the Riesz potential of $f$ is the function $I_\alpha(f)$ given by 
$$I_\alpha(f)(x) = \int_{\mathbb{R}^d} \dfrac{f(y)}{|x-y|^{d-\alpha}} dy,  $$
with $0 < \alpha < d.$
Regarding these potentials one has the following result \cite[Chapter V, Section 1]{stein}:

\begin{proposition} \label{steinprop}
 Let $0 < \alpha < d$ and $1 < p < d/\alpha.$ If $f \in L^p(\mathbb{R}^d),$ then $I_\alpha(f)(x)$ converges absolutely for a.e. $x \in \mathbb{R}^d$ and 
 $$||I_\alpha(|f|)||_{L^{\frac{dp}{d-\alpha p}}(\mathbb{R}^d)} \leq C || f ||_{L^p(\mathbb{R}^d)},$$
 for some positive constant $C = C(\alpha, d,p).$
\end{proposition}
Combining this previous proposition with the properties of the Neumann function one has: 
\begin{proposition} \label{neumannprop}
 Let $d \in \mathbb{N} \setminus \{1,2\},$ $f,g \in  L^\gamma(\Omega), \ \phi = N * f, \ \varphi = N * g, \ \nabla \phi = \nabla_xN * f,$ and $\nabla \varphi = \nabla_x N * g,$ where $\gamma \geq \frac{2d}{d+2},$ $\Omega \subseteq \mathbb{R}^d$ is a smooth bounded domain with smooth boundary, and $N$ is the Neumann function. Then, $\phi, \varphi \in L^{\frac{2d}{d-2}}(\Omega),$ $\nabla \phi, \nabla \varphi \in L^2(\Omega),$ and 
 \begin{equation} \label{intparts} 
 \int_\Omega \nabla \phi \cdot \nabla \varphi dx = \int_\Omega f \varphi dx = \int_\Omega g \phi dx = \int_\Omega \int_\Omega f(x)N(x,y)g(y) dxdy.
\end{equation}
\end{proposition}
\begin{proof}
First one demonstrates that $\phi \in L^{\frac{2d}{d-2}}(\Omega)$ and $\nabla \phi \in L^2(\Omega)$ (so $\varphi \in L^{\frac{2d}{d-2}}(\Omega)$ and $\nabla \varphi \in L^2(\Omega)$ aswell).
 Set $p = \frac{2d}{d+2}$ and observe that since $d > 2$ one has $1< p < d/2$.  Let $\tilde{f}$ be given by 
 $$\tilde{f}(x) = 
   \begin{cases}
    f(x), \ \text{if} \ x \in \Omega, \\
    0, \ \text{otherwise.}
   \end{cases}$$
Clearly $\tilde{f} \in L^1(\mathbb{R}^d) \cap L^\gamma(\mathbb{R}^d),$ and since $\gamma \geq p = \frac{2d}{d+2}$ interpolation gives that $\tilde{f} \in L^p(\mathbb{R}^d).$  
From the properties of the Neumann function one deduces that 

\begin{equation*}
 |\phi(x)|  \leq \int_\Omega |N(x,y)||f(y)|dy 
             \leq C \int_\Omega \frac{|f(y)|}{|x-y|^{d-2}}dy 
              \leq C \int_{\mathbb{R}^d} \frac{|\tilde{f}(y)|}{|x-y|^{d-2}}dy 
              = CI_2(|\tilde{f}|)(x), \ x \in \Omega.
\end{equation*}
\\
Using Proposition \ref{steinprop} with $\alpha = 2$, $p = \frac{2d}{d+2}$,  one obtains

\begin{equation*}
 ||\phi||_{L^{\frac{2d}{d-2}}(\Omega)}  \leq C||I_2(|\tilde{f}|)||_{L^{\frac{2d}{d-2}}(\Omega)} \leq C||I_2(|\tilde{f}|)||_{L^{\frac{2d}{d-2}}(\mathbb{R}^d)} \leq C || \tilde{f} ||_{L^p(\mathbb{R}^d)} 
              = C || f ||_{L^p(\Omega)}.
 \end{equation*}
                                                      
Similarly, $$|\nabla \phi(x)| \leq C I_1(|\tilde{f}|)(x), \ x \in \Omega, $$
 hence 
 $$
 ||\nabla \phi||_{L^2(\Omega)} \leq  C \| I_1 ( |\tilde{f}| ) \|_{L^2(\Omega)} \le  C ||f||_{L^p(\Omega)}, 
 $$
 where we used Proposition \ref{steinprop} with $\alpha = 1$, $p = \frac{2d}{d+2}$.  \par 
 To prove the second and third equalities of expression (\ref{intparts}) one observes that $p^\prime = \frac{2d}{d-2},$ so 
 $$\int_\Omega  f \varphi dx \leq \Big( \int_\Omega |f|^{p}dx \Big)^{\frac{1}{p}} \Big( \int_\Omega |\varphi|^{p^\prime}dx \Big)^{\frac{1}{p^\prime}} < + \infty, $$ \\
 then Fubini's theorem and the symmetry of the Neumann function yield the desired conclusion. 
  
 Next, we prove the first equality in (\ref{intparts}). For $f,g \in L^p(\Omega),$ there exist sequences $(f_n)_{n \in \mathbb{N}},(g_n)_{n \in \mathbb{N}}$ belonging to $ C_c^{\infty}(\Omega)$ such that $$f_n \to f \ \text{in} \ L^p(\Omega), $$ $$g_n \to g \ \text{in} \ L^p(\Omega). $$ Let $\phi_n = N * f_n$ and $\varphi_n = N * g_n $ Then,  
 $$||\phi - \phi_n ||_{L^{p^\prime}(\Omega)} \leq C ||f - f_n ||_{L^p(\Omega)} \to 0 \ \text{as} \ n \to +\infty, $$
 and
 $$||\nabla \phi - \nabla \phi_n ||_{L^2(\Omega)} \leq C ||f - f_n ||_{L^p(\Omega)} \to 0 \ \text{as} \ n \to +\infty. $$
 In other words, $\phi_n \to \phi \ \text{in} \ L^{p^\prime}(\Omega)$ and $\nabla \phi_n \to \nabla \phi \ \text{in} \ L^2(\Omega),$ and the same holds for $\varphi_n, \varphi, \nabla \varphi_n, \nabla \varphi.$ \\
 Thus,
 \begin{equation*}
  \begin{split}
   \Big|\int_\Omega  f_n \varphi_n dx - \int_\Omega f \varphi dx\Big|  & \leq \int_\Omega |f_n| |\varphi_n - \varphi|dx + \int_\Omega  |f_n - f||\varphi| dx \\
   & \leq ||f_n ||_{L^p(\Omega)} ||\varphi_n - \varphi ||_{L^{p^{\prime}}(\Omega)} +  ||f_n - f||_{L^p(\Omega)}|| \varphi ||_{L^{p^{\prime}}(\Omega)} \\
   & \to 0 \ \text{as} \ n \to +\infty,
  \end{split}
\end{equation*}
and 
\begin{equation*}
\begin{split}
  \Big| \int_\Omega \nabla \phi_n \cdot \nabla \varphi_n dx - \int_\Omega \nabla \phi \cdot \nabla \varphi dx \Big| & \leq ||\nabla \phi_n - \nabla \phi ||_{L^2(\Omega)}||\nabla \varphi_n ||_{L^2(\Omega)}+|| \nabla \phi||_{L^2(\Omega)}||\nabla \varphi_n - \nabla \varphi ||_{L^2(\Omega)} \\
  & \to 0 \ \text{as} \ n \to +\infty.
\end{split}
\end{equation*}
 Observing that $f_n, \phi_n, \varphi_n$ satisfy 
$$ \int_\Omega \nabla \phi_n \cdot \nabla \varphi_ndx = \int_\Omega f_n \varphi_n dx, $$
after letting $n \to +\infty$ one obtains the desired identity. 
\end{proof}

\par
We finish this subsection with a result proved in \cite[Lemma 2.4]{lattanziothanos2013}, which is used in the proof of Lemma \ref{lemmaJ3}.
\begin{lemma} \label{hlemma}
 Let $h \in C^2(]0,+\infty[) \cap C([0,+\infty[)$ be such that $\lim\limits_{r \to +\infty} \frac{h(r)}{r^\gamma}=\frac{k}{\gamma-1}$ for some $k>0$ and $\gamma > 1,$ and $h^{\prime \prime}(r) > 0 \ \forall r>0.$ Assume that $\bar{r} \in [\delta,M],$ where $\delta > 0$ and $M < +\infty.$ Then, there exists $R \geq M+1$ and positive constants $C_1, C_2$ such that 
\begin{equation*} 
 h(r| \bar{r}) \geq
 \begin{cases}
  C_1 |r-\bar{r}|^2, \ \text{if} \ (r,\bar{r}) \in [0,R]\times [\delta,M] \\
  
  C_2 |r-\bar{r}|^{\gamma}, \ \text{if} \ (r,\bar{r}) \in ]R,+\infty [\times [\delta,M].
  
 \end{cases}
\end{equation*}
Furthermore, if $\gamma \geq 2$, then $h(r| \bar{r}) \geq C|r-\bar{r}|^2 $ for every $(r,\bar{r})\in[0,+\infty[\times[\delta,M]$, where $C=\min\{C_1,C_2\}.$ 
\end{lemma} 

\subsection{Derivation of the relative energy inequality}
The relative energy inequality is now derived within the regularity class detailed in section \ref{sec4}. 

\begin{proposition} \label{relativeentropy}
Let $(\rho,\rho u,n,nv),$ with $\phi = N*(\rho - n),$ be a dissipative weak solution of (\ref{BEP}) with $\gamma_1,\gamma_2 \geq \frac{2d}{d+2}$, and let $(\bar{\rho},\bar{\rho} \bar{u},\bar{n},\bar{n} \bar{v}),$ with $\bar{\phi}= N*(\bar{\rho}-\bar{n}),$ be a strong and bounded away from vacuum solution of (\ref{BDD}). Then, for each $t \in [0,T[$, the relative energy $\Psi$ between these two solutions satisfies the following relative energy inequality: 
\begin{equation} \label{relativeentropyinequality}
\Psi(t)-\Psi(0) + \int_0^t \int_\Omega \rho |u - \bar{u}|^2 + n |v - \bar{v}|^2 dx d\tau   \leq \mathcal{J}_1(t) + \mathcal{J}_2(t) + \mathcal{J}_3(t) + \mathcal{J}_4(t),
\end{equation}
where 
\begin{equation*}
\begin{split}
\mathcal{J}_1(t) & = -\varepsilon \int_0^t \int_\Omega \nabla \bar{u}: \rho (u - \bar{u}) \otimes (u - \bar{u})+\nabla \bar{v}: n (v - \bar{v}) \otimes (v - \bar{v})dxd \tau, \\
\mathcal{J}_2(t) & = -  \int_0^t \int_\Omega (\nabla \cdot \bar{u}) p_1(\rho | \bar{\rho}) + (\nabla \cdot \bar{v}) p_2(n | \bar{n}) dx d\tau, \\
\mathcal{J}_3(t) & =  \int_0^t \int_\Omega \big((\rho - \bar{\rho})\bar{u}-(n - \bar{n})\bar{v}\big)\cdot \nabla (\phi - \bar{\phi})  dx d\tau, \\
\mathcal{J}_4(t) & = -\varepsilon \int_0^t \int_\Omega \frac{\rho}{\bar{\rho}} \bar{e}_1 \cdot (u - \bar{u}) + \frac{n}{\bar{n}} \bar{e}_2 \cdot (v - \bar{v}) dx d\tau. \\
\end{split}
\end{equation*}
\end{proposition}
\begin{proof}
Fix $t \in [0,T[,$ let $\kappa$ be such that $t+\kappa < T$, and define $\theta : [0,T[\  \to \mathbb{R}$ by 
$$\theta (\tau) = 
\begin{dcases}
1, \ \text{if} \ 0 \leq \tau < t \\ 
\frac{t- \tau}{\kappa}+1, \ \text{if} \ t \leq \tau < t+\kappa \\
0, \ \text{if} \ t+\kappa \leq \tau < T. 
\end{dcases}
 $$ 
 Using this choice of $\theta$ in (\ref{weakdissip}) it yields 
 \begin{equation*}
 \begin{split}
  & \int_t^{t+ \kappa} \int_\Omega \frac{1}{\kappa}\big(\varepsilon \tfrac{1}{2}\rho|u|^2+\varepsilon \tfrac{1}{2}n|v|^2+h_1(\rho)+h_2(n)+ \tfrac{1}{2} |\nabla \phi|^2\big)dxd\tau \\
 & + \int_0^t \int_\Omega \rho|u|^2+n|v|^2dxd\tau + \int_t^{t+ \kappa} \int_\Omega \Big(\frac{t - \tau}{\kappa} + 1\Big) \big(\rho|u|^2+n|v|^2\big)dxd\tau \\
 & \leq  \int_\Omega \big(\varepsilon\tfrac{1}{2}\rho|u|^2+\varepsilon\tfrac{1}{2}n|v|^2+h_1(\rho)+h_2(n) + \tfrac{1}{2} |\nabla \phi|^2 \big)\big|_{\tau = 0}dx.
 \end{split}
\end{equation*}
Letting $\kappa \to 0^+$ above one deduces
\begin{equation} \label{REweakdiss}
 \begin{split}
  & \int_\Omega \big(\varepsilon \tfrac{1}{2}\rho|u|^2+\varepsilon \tfrac{1}{2}n|v|^2+h_1(\rho)+h_2(n)+ \tfrac{1}{2} |\nabla \phi|^2\big)\big|_{\tau = 0}^{\tau = t}dx \\
  & \leq  - \int_0^t \int_\Omega \rho|u|^2+n|v|^2dxd\tau.
 \end{split}
\end{equation}

Next, observe that from \eqref{BDDlifted}, after a straightforward calculation, one obtains
\begin{equation} \label{ue}
 \begin{split}
   & \int_0^t \int_\Omega \bar{u} \cdot \bar{e}_1 \ dx d\tau = \int_\Omega \tfrac{1}{2} \bar{\rho} |\bar{u}|^2  dx \Big |_{\tau=0}^{\tau=t},
   \\ 
  & \int_0^t \int_\Omega \bar{v} \cdot \bar{e}_2 \ dx d\tau = \int_\Omega \tfrac{1}{2} \bar{n} |\bar{v}|^2  dx \Big |_{\tau=0}^{\tau=t} .
 \end{split}
\end{equation}
Using the same choice of $\theta$ in (\ref{strongdissip}) together with (\ref{ue}) gives
\begin{equation} \label{REstrongdiss}
 \begin{split}
  & \int_\Omega \big(\varepsilon \tfrac{1}{2}\bar{\rho}|\bar{u}|^2+\varepsilon \tfrac{1}{2}\bar{n}|\bar{v}|^2+h_1(\bar{\rho})+h_2(\bar{n})+ \tfrac{1}{2} |\nabla \bar{\phi}|^2\big)\big|_{\tau = 0}^{\tau = t}dx \\
  & =  - \int_0^t \int_\Omega \bar{\rho}|\bar{u}|^2+\bar{n}|\bar{v}|^2dxd\tau + \varepsilon \int_0^t \int_\Omega \bar{u} \cdot \bar{e}_1 + \bar{v} \cdot \bar{e}_2 \ dxd\tau.
 \end{split}
\end{equation}
\par 
Regarding the difference $(\rho - \bar{\rho},\rho u - \bar{\rho} \bar{u},n - \bar{n},nv -\bar{n} \bar{v}) $ between a weak solution of (\ref{BEP}) and a strong solution of (\ref{BDD}), one has the following: 
\begin{equation*} 
         -\int_0^T \int_\Omega \varphi_t (\rho - \bar{\rho}) \ dxdt-\int_0^T \int_\Omega \nabla \varphi \cdot (\rho u - \bar{\rho} \bar{u}) dxdt - \int_\Omega \varphi (\rho - \bar{\rho}) \big|_{t=0}dx=0,
\end{equation*}
\begin{equation*} 
        \begin{split}
        -& \varepsilon \int_0^T \int_\Omega \tilde{\varphi}_t \cdot (\rho u- \bar{\rho} \bar{u})dxdt - \varepsilon \int_0^T \int_\Omega \nabla \tilde{\varphi} : (\rho u \otimes u - \bar{\rho} \bar{u} \otimes \bar{u})dx dt \\
        &-  \int_0^T \int_\Omega (\nabla \cdot \tilde{\varphi}) \big(p_1(\rho)-p_1(\bar{\rho}) \big)dxdt  
        - \varepsilon \int_\Omega \tilde{\varphi} \cdot (\rho u - \bar{\rho} \bar{u})\big|_{t=0}dx \\
        =& \  -  \int_0^T \int_\Omega \tilde{\varphi} \cdot (\rho \nabla \phi - \bar{\rho} \nabla \bar{\phi})dxdt - \int_0^T \int_\Omega \tilde{\varphi} \cdot (\rho u - \bar{\rho} \bar{u})dxdt- \varepsilon \int_0^T \int_\Omega \tilde{\varphi} \cdot \bar{e}_1dxdt,
        \end{split}
\end{equation*}
\begin{equation*} 
           -\int_0^T \int_\Omega \psi_t (n - \bar{n}) \ dxdt-\int_0^T \int_\Omega \nabla \psi \cdot (nv - \bar{n} \bar{v}) dxdt - \int_\Omega \psi (n - \bar{n}) \big|_{t=0}dx=0,
\end{equation*}
\begin{equation*} 
        \begin{split}
         -& \varepsilon \int_0^T \int_\Omega \tilde{\psi}_t \cdot (nv - \bar{n} \bar{v})dxdt - \varepsilon \int_0^T \int_\Omega \nabla \tilde{\psi} :(  nv \otimes v - \bar{n} \bar{v} \otimes \bar{v}) dx dt\\
         & -  \int_0^T \int_\Omega (\nabla \cdot \tilde{\psi}) \big(p_2(n)-p_2(\bar{n}) \big) dxdt 
         - \varepsilon \int_\Omega \tilde{\psi} \cdot (nv- \bar{n} \bar{v})\big|_{t=0}dx \\
         = & \  \int_0^T \int_\Omega \tilde{\psi} \cdot (n \nabla \phi - \bar{n} \nabla \bar{\phi})dxdt - \int_0^T \int_\Omega \tilde{\psi} \cdot (nv - \bar{n} \bar{v})dxdt- \varepsilon \int_0^T \int_\Omega \tilde{\psi} \cdot \bar{e}_2dxdt,
        \end{split}
\end{equation*}  
        for all Lipschitz test functions $\varphi, \psi : [0,T[ \times \Omega \to \mathbb{R}, \ \tilde{\varphi}, \tilde{\psi}:[0,T[ \times \Omega \to \mathbb{R}^d$ compactly supported in time and with $\tilde{\varphi}, \tilde{\psi}$ satisfying the no-flux boundary condition on the boundary. \par 
        Set
        $$(\varphi, \ \tilde{\varphi}, \ \psi, \ \tilde{\psi}) =\big( \theta (- \varepsilon \tfrac{1}{2}|\bar{u}|^2+h_1^\prime(\bar{\rho})+ \bar{\phi} ), \  \theta \bar{u},\ \theta (- \varepsilon \tfrac{1}{2}|\bar{v}|^2+h_2^\prime(\bar{n})- \bar{\phi}), \ \theta \bar{v} \big), $$
        where $\theta$ is as before. 
        In view of \eqref{boundarycondBDD}, this choice of $(\varphi, \tilde{\varphi}, \psi, \tilde{\psi})$ satisfies  
        $\tilde{\varphi}\cdot \nu = \tilde{\psi}\cdot \nu = 0$ on $[0,T[ \times \partial \Omega$   and can be used in the weak formulation.
         Using that choice and  letting $\kappa \to 0^+$ one obtains 
\begin{equation} \label{weakdiff1}
    \begin{split}
      & \int_\Omega \big( (-\varepsilon \tfrac{1}{2}|\bar{u}|^2+h_1^\prime(\bar{\rho})+ \bar{\phi} )(\rho - \bar{\rho}) \big)\big|_{\tau = 0}^{\tau = t}dx -\int_0^t \int_\Omega \partial_\tau (- \varepsilon \tfrac{1}{2}|\bar{u}|^2+ h_1^\prime(\bar{\rho})+ \bar{\phi} )(\rho - \bar{\rho})dxd\tau \\
     & -  \int_0^t \int_\Omega \nabla (- \varepsilon \tfrac{1}{2}|\bar{u}|^2+h_1^\prime(\bar{\rho})+ \bar{\phi} ) \cdot (\rho u - \bar{\rho} \bar{u})dxd\tau 
     = 0,
    \end{split}
\end{equation}
\begin{equation} \label{weakdiff2}
    \begin{split}
       \varepsilon & \int_\Omega \big( \bar{u} \cdot (\rho u - \bar{\rho} \bar{u}) \big) \big|_{\tau = 0}^{\tau = t}dx - \varepsilon \int_0^t \int_\Omega (\partial_\tau \bar{u}) \cdot (\rho u - \bar{\rho} \bar{u})dxd\tau \\
       & -  \varepsilon \int_0^t \int_\Omega \nabla \bar{u} : (\rho u \otimes u - \bar{\rho} \bar{u} \otimes \bar{u})dxd\tau- \int_0^t \int_\Omega (\nabla \cdot \bar{u})\big(p_1(\rho)-p_1(\bar{\rho})\big)dxd\tau \\ 
        = & - \int_0^t \int_\Omega \bar{u} \cdot (\rho \nabla \phi - \bar{\rho} \nabla \bar{\phi})dxd\tau - \int_0^t \int_\Omega \bar{u} \cdot (\rho u - \bar{\rho} \bar{u})dxd\tau - \varepsilon \int_0^t \int_\Omega \bar{u} \cdot \bar{e}_1dxd\tau,
    \end{split}
\end{equation}
\begin{equation} \label{weakdiff3}
    \begin{split}
      &\int_\Omega \big( (- \varepsilon \tfrac{1}{2}|\bar{v}|^2+ h_2^\prime(\bar{n})- \bar{\phi} )(n - \bar{n}) \big)\big|_{\tau = 0}^{\tau = t}dx -\int_0^t \int_\Omega \partial_\tau (- \varepsilon \tfrac{1}{2}|\bar{v}|^2+h_2^\prime(\bar{n})- \bar{\phi} )(n - \bar{n})dxd\tau \\
     & -  \int_0^t \int_\Omega \nabla (- \varepsilon \tfrac{1}{2}|\bar{v}|^2+h_2^\prime(\bar{n})- \bar{\phi} ) \cdot (nv - \bar{n} \bar{v})dxd\tau = 0,
   \end{split}
\end{equation}
\begin{equation} \label{weakdiff4}
    \begin{split}
       \varepsilon &\int_\Omega \big( \bar{v} \cdot (nv - \bar{n} \bar{v})\big) \big|_{\tau = 0}^{\tau = t}dx - \varepsilon \int_0^t \int_\Omega (\partial_\tau \bar{v}) \cdot (nv - \bar{n} \bar{v})dxd\tau \\
       & - \varepsilon \int_0^t \int_\Omega \nabla \bar{v} : (nv \otimes v - \bar{n} \bar{v} \otimes \bar{v})dxd\tau- \int_0^t \int_\Omega (\nabla \cdot \bar{v})\big(p_2(n)-p_2(\bar{n})\big)dxd\tau \\ 
        = &  \int_0^t \int_\Omega \bar{v} \cdot (\rho \nabla \phi - \bar{\rho} \nabla \bar{\phi})dxd\tau - \int_0^t \int_\Omega \bar{v} \cdot (nv - \bar{n} \bar{v})dxd\tau - \varepsilon \int_0^t \int_\Omega \bar{v} \cdot \bar{e}_2dxd\tau.
    \end{split}
\end{equation} 
From the computation $(\ref{REweakdiss}) - (\ref{REstrongdiss}) - \big((\ref{weakdiff1})+(\ref{weakdiff2})+(\ref{weakdiff3})+(\ref{weakdiff4})\big)$ it follows that 
\begin{equation} \label{RE3}
 \begin{split}
 \int_\Omega &  \big(\varepsilon\tfrac{1}{2} \rho |u - \bar{u}|^2 + \varepsilon \tfrac{1}{2} n |v - \bar{v}|^2 + h_1(\rho | \bar{\rho}) + h_2(n | \bar{n}) + \tfrac{1}{2}|\nabla(\phi - \bar{\phi})|^2\big) \big|_{\tau = 0}^{\tau = t}dx \\
 \leq & - \int_0^t \int_\Omega \rho |u|^2- \bar{\rho} |\bar{u}|^2 - \bar{u} \cdot (\rho u - \bar{\rho} \bar{u}) dxd\tau \\ 
 & -  \int_0^t \int_\Omega n |v|^2- \bar{n} |\bar{v}|^2 - \bar{v} \cdot (n v - \bar{n} \bar{v}) dxd\tau \\ 
& - \int_0^t \int_\Omega \partial_\tau (-\varepsilon \tfrac{1}{2}|\bar{u}|^2 + h_1^\prime(\bar{\rho}) + \bar{\phi} )(\rho - \bar{\rho})dxd\tau \\
& - \int_0^t \int_\Omega \partial_\tau (-\varepsilon \tfrac{1}{2}|\bar{v}|^2+h_2^\prime(\bar{n}) -\bar{\phi} )(n - \bar{n})dxd\tau \\
& -\varepsilon \int_0^t \int_\Omega (\partial_\tau \bar{u}) \cdot (\rho u - \bar{\rho} \bar{u})dxd\tau -\varepsilon \int_0^t \int_\Omega (\partial_\tau \bar{v}) \cdot (nv - \bar{n} \bar{v})dxd\tau \\
& - \int_0^t \int_\Omega \nabla (-\varepsilon \tfrac{1}{2}|\bar{u}|^2+h_1^\prime(\bar{\rho}) +\bar{\phi} ) \cdot (\rho u - \bar{\rho} \bar{u})dxd\tau \\ 
 & - \int_0^t \int_\Omega \nabla (-\varepsilon \tfrac{1}{2}|\bar{v}|^2+h_2^\prime(\bar{n}) -\bar{\phi} ) \cdot (nv - \bar{n} \bar{v})dxd\tau \\
& - \varepsilon \int_0^t \int_\Omega \nabla \bar{u} : (\rho u \otimes u - \bar{\rho} \bar{u} \otimes \bar{u})dxd\tau- \varepsilon \int_0^t \int_\Omega \nabla \bar{v} : (n v \otimes v - \bar{n} \bar{v} \otimes \bar{v})dxd\tau \\
& - \int_0^t \int_\Omega (\nabla \cdot \bar{u})\big(p_1(\rho)-p_1(\bar{\rho})\big)dxd\tau - \int_0^t \int_\Omega (\nabla \cdot \bar{v})\big(p_2(n)-p_2(\bar{n})\big)dxd\tau 
\\
& +  \int_0^t \int_\Omega \bar{u} \cdot (\rho \nabla \phi - \bar{\rho} \nabla \bar{\phi})dxd\tau -  \int_0^t \int_\Omega \bar{v} \cdot (n \nabla \phi - \bar{n} \nabla \bar{\phi})dxd\tau.
  \end{split}
\end{equation} \par
The strong, bounded away from vacuum solution $(\bar{\rho},\bar{\rho} \bar{u},\bar{n},\bar{n} \bar{v}),$ with $ \bar \phi = N *(\bar \rho - \bar n),$ satisfies the following system 
\begin{equation} \label{uvbarsystem}
 \begin{dcases}
  \varepsilon \big( \bar{u}_t+\bar{u} \cdot \nabla \bar{u} \big) = -\nabla \big(h_1^\prime(\bar{\rho})+\bar{\phi}\big)- \bar{u} + \varepsilon \frac{\bar{e}_1}{\bar{\rho}} 
 \\
  \varepsilon \big( \bar{v}_t+\bar{v} \cdot \nabla \bar{v} \big)= - \nabla \big(h_2^\prime(\bar{n})-\bar{\phi}\big)- \bar{v} + \varepsilon \frac{\bar{e}_2}{\bar{n}}.
 \end{dcases}
\end{equation} 
 Multiplying the first and second equations above by $\rho (u - \bar{u})$ and $n (v - \bar{v})$, respectively, yields: 
\begin{equation} \label{RE4}
 \begin{dcases} 
  \begin{split}
    \varepsilon  &\big(- \tfrac{1}{2} |\bar{u}|^2\big)_t(\rho - \bar{\rho})+\varepsilon \bar{u}_t \cdot (\rho u - \bar{\rho} \bar{u}) + \varepsilon \nabla \big(- \tfrac{1}{2} |\bar{u}|^2\big) \cdot (\rho u - \bar{\rho} \bar{u}) \\
     & + \varepsilon \nabla \bar{u} : (\rho u \otimes u-\bar{\rho} \bar{u} \otimes \bar{u}) \\ 
   = &  -\rho\nabla h_1^\prime(\bar{\rho}) \cdot (u-\bar{u})-\rho \nabla \bar{\phi} \cdot (u-\bar{u})-\rho \bar{u} \cdot (u - \bar{u})\\
   & + \varepsilon \rho \nabla \bar{u} : (u - \bar{u}) \otimes (u - \bar{u})+ \varepsilon \frac{\rho}{\bar{\rho}}\bar{e}_1 \cdot (u- \bar{u})
  \end{split}
 \\ \\
    \begin{split}
     \varepsilon &\big(- \tfrac{1}{2} |\bar{v}|^2\big)_t(n - \bar{n})+ \varepsilon \bar{v}_t \cdot (n v - \bar{v} \bar{v}) + \varepsilon \nabla\big(- \tfrac{1}{2} |\bar{v}|^2\big) \cdot (n v - \bar{n} \bar{v})\\
     & +\nabla \bar{v} : (n v \otimes v-\bar{n} \bar{v} \otimes \bar{v}) \\ 
   = &  -n\nabla h_2^\prime(\bar{n}) \cdot (v-\bar{v})+n \nabla \bar{\phi} \cdot (v-\bar{v})-n \bar{v} \cdot (v - \bar{v})\\
    & + \varepsilon  \nabla \bar{v} : (v - \bar{v}) \otimes (v - \bar{v})+\varepsilon \frac{n}{\bar{n}}\bar{e}_2 \cdot (v- \bar{v}).
  \end{split}
\end{dcases}
\end{equation} 
Substituting (\ref{RE4}) into (\ref{RE3}) renders that \\ 
\begin{equation} \label{RE5}
 \begin{split}
  \int_\Omega & \big( \varepsilon \tfrac{1}{2} \rho |u - \bar{u}|^2+\varepsilon \tfrac{1}{2} n |v - \bar{v}|^2 +  h_1(\rho | \bar{\rho}) +  h_2(n | \bar{n}) + \tfrac{1}{2} |\nabla (\phi - \bar{\phi})|^2 \big)  \big|_{\tau=0}^{\tau=t}dx \\
   \leq & -  \int_0^t \int_\Omega \rho |u|^2- \bar{\rho} |\bar{u}|^2 - \bar{u} \cdot (\rho u - \bar{\rho} \bar{u}) - \rho \bar{u} \cdot (u - \bar{u})dxd\tau \\ 
   & -  \int_0^t \int_\Omega n |v|^2- \bar{n} |\bar{v}|^2 - \bar{v} \cdot (n v - \bar{n} \bar{v})- n \bar{v} \cdot (v - \bar{v})  dxd\tau \\ 
   & -\varepsilon \int_0^t \int_\Omega \nabla \bar{u} : \rho (u-\bar{u}) \otimes (u - \bar{u}) + \nabla \bar{v} : n (v-\bar{v}) \otimes (v - \bar{v})dxd\tau \\
   & -  \int_0^t \int_\Omega \partial_\tau \big(h_1^\prime(\bar{\rho}) \big)(\rho - \bar{\rho}) + \nabla h_1^\prime(\bar{\rho}) \cdot (\rho u -\bar{\rho} \bar{u}) dxd\tau \\
   & -  \int_0^t \int_\Omega (\nabla \cdot \bar{u}) \big(p_1(\rho) - p_1(\bar{\rho}) \big) - \nabla h_1^\prime(\bar{\rho}) \cdot (\rho u -\rho \bar{u}) dxd\tau \\
   & -  \int_0^t \int_\Omega \partial_\tau \big(h_2^\prime(\bar{n}) \big)(n - \bar{n}) + \nabla h_2^\prime(\bar{n}) \cdot (n v -\bar{n} \bar{v}) dxd\tau \\
   & -  \int_0^t \int_\Omega (\nabla \cdot \bar{u}) \big(p_2(n) - p_2(\bar{n}) \big) - \nabla h_2^\prime(\bar{n}) \cdot (n v -n \bar{v}) dxd\tau \\
   & +  \int_0^t \int_\Omega - (\partial_\tau \bar{\phi})(\rho - \bar{\rho}) - \nabla \bar{\phi} \cdot (\rho u - \bar{\rho} \bar{u}) dx d\tau \\
   & +  \int_0^t \int_\Omega \bar{u} \cdot (\rho \nabla \phi - \bar{\rho} \nabla \bar{\phi}) + \rho \nabla \bar{\phi} \cdot (u - \bar{u}) dx d\tau \\  
    & -  \int_0^t \int_\Omega - (\partial_\tau \bar{\phi})(n - \bar{n}) - \nabla \bar{\phi} \cdot (n v - \bar{n} \bar{v}) dx d\tau \\
   & -  \int_0^t \int_\Omega \bar{v} \cdot (\rho \nabla \phi - \bar{n} \nabla \bar{\phi}) + n \nabla \bar{\phi} \cdot (v - \bar{v}) dx d\tau \\
   & - \varepsilon \int_0^t \int_\Omega \frac{\rho}{\bar{\rho}} \bar{e}_1 \cdot (u - \bar{u}) + \frac{n}{\bar{n}} \bar{e}_2 \cdot (v - \bar{v}) dx d\tau.
 \end{split}
\end{equation} 
A simple calculation provides
\begin{equation} \label{RE6}
 -  \int_0^t \int_\Omega \rho |u|^2- \bar{\rho} |\bar{u}|^2 - \bar{u} \cdot (\rho u - \bar{\rho} \bar{u}) - \rho \bar{u} \cdot (u - \bar{u})dxd\tau = -  \int_0^t \int_\Omega \rho |u - \bar{u}|^2 dxd \tau,
\end{equation}
\begin{equation} \label{RE7}
 -  \int_0^t \int_\Omega n |v|^2- \bar{n} |\bar{v}|^2 - \bar{v} \cdot (n v - \bar{n} \bar{v})- n \bar{v} \cdot (v - \bar{v})  dxd\tau = -  \int_0^t \int_\Omega n |v - \bar{v}|^2 dxd \tau.
\end{equation} 
Additionally, since $\bar{\rho}_t+\nabla \cdot (\bar{\rho} \bar{u}) = 0$ and $\bar{n}_t+\nabla \cdot (\bar{n} \bar{v}) = 0$ one derives
\begin{equation} \label{RE8}
\begin{split}
  - &  \int_0^t \int_\Omega \partial_\tau \big(h_1^\prime(\bar{\rho}) \big)(\rho - \bar{\rho}) + \nabla h_1^\prime(\bar{\rho}) \cdot (\rho u -\bar{\rho} \bar{u}) dxd\tau \\
   & -  \int_0^t \int_\Omega (\nabla \cdot \bar{u}) \big(p_1(\rho) - p_1(\bar{\rho}) \big) - \nabla h_1^\prime(\bar{\rho}) \cdot (\rho u -\rho \bar{u}) dxd\tau\\
    = & -  \int_0^t \int_\Omega (\nabla \cdot \bar{u}) p_1(\rho | \bar{\rho})dx d\tau,
\end{split}
\end{equation}
\begin{equation} \label{RE9}
\begin{split}
 - &  \int_0^t \int_\Omega \partial_\tau \big( h_2^\prime(\bar{n}) \big) (n - \bar{n}) + \nabla h_2^\prime(\bar{n}) \cdot (n v -\bar{n} \bar{v}) dxd\tau \\
   & -  \int_0^t \int_\Omega (\nabla \cdot \bar{v}) \big(p_2(n) - p_2(\bar{n}) \big) - \nabla h_2^\prime(\bar{n}) \cdot (n v -\rho \bar{v}) dxd\tau \\
    = & -  \int_0^t \int_\Omega (\nabla \cdot \bar{v}) p_2(n | \bar{n})dx d\tau.
    \end{split}
\end{equation}
Moreover, the second equality of identity (\ref{intparts}) and the no-flux boundary conditions (\ref{nofluxrhou}) imply that
\begin{equation} \label{RE10}
 \begin{split}
   &   \int_0^t \int_\Omega - (\partial_\tau \bar{\phi})(\rho - \bar{\rho}) - \nabla \bar{\phi} \cdot (\rho u - \bar{\rho} \bar{u}) dx d\tau \\
   & +  \int_0^t \int_\Omega \bar{u} \cdot (\rho \nabla \phi - \bar{\rho} \nabla \bar{\phi}) + \rho \nabla \bar{\phi} \cdot (u - \bar{u}) dx d\tau \\  
    & -  \int_0^t \int_\Omega - (\partial_\tau \bar{\phi})(n - \bar{n}) - \nabla \bar{\phi} \cdot (n v - \bar{n} \bar{v}) dx d\tau \\
   & -  \int_0^t \int_\Omega \bar{v} \cdot (\rho \nabla \phi - \bar{n} \nabla \bar{\phi}) + n \nabla \bar{\phi} \cdot (v - \bar{v}) dx d\tau \\
   = &   \int_0^t \int_\Omega - (\rho - \bar{\rho} - n + \bar{n})(\partial_\tau \bar{\phi}) + \nabla(\phi - \bar{\phi}) \cdot (\rho \bar{u} - n \bar{v})  dx d\tau \\
   = &  \int_0^t \int_\Omega \big((\rho - \bar{\rho})\bar{u} - (n - \bar{n})\bar{v} \big) \cdot \nabla (\phi - \bar{\phi}) dx d\tau.
 \end{split}
\end{equation}
Finally, replacing (\ref{RE6}), (\ref{RE7}), (\ref{RE8}), (\ref{RE9}) and (\ref{RE10})  in (\ref{RE5}) yields \\
\begin{equation*}
 \begin{split}
   \int_\Omega & \big( \varepsilon \tfrac{1}{2} \rho |u - \bar{u}|^2+\varepsilon \tfrac{1}{2} n |v - \bar{v}|^2 +  h_1(\rho | \bar{\rho}) +  h_2(n | \bar{n}) + \tfrac{1}{2} |\nabla (\phi - \bar{\phi})|^2 \big)  \big|_{\tau=0}^{\tau=t}dx \\
    \leq & -  \int_0^t \int_\Omega \rho |u - \bar{u}|^2 + n |v - \bar{v}|^2 dx d\tau \\
    & - \varepsilon \int_0^t \int_\Omega \nabla \bar{u}: \rho (u - \bar{u}) \otimes (u - \bar{u})+\nabla \bar{v}: n (v - \bar{v}) \otimes (v - \bar{v})dxd \tau \\
    & -   \int_0^t \int_\Omega (\nabla \cdot \bar{u}) p_1(\rho | \bar{\rho}) + (\nabla \cdot \bar{v}) p_2(n | \bar{n}) dx d\tau \\
    & +  \int_0^t \int_\Omega \big((\rho - \bar{\rho})\bar{u}-(n - \bar{n})\bar{v}\big)\cdot \nabla (\phi - \bar{\phi})  dx d\tau \\
   & - \varepsilon \int_0^t \int_\Omega \frac{\rho}{\bar{\rho}} \bar{e}_1 \cdot (u - \bar{u}) + \frac{n}{\bar{n}} \bar{e}_2 \cdot (v - \bar{v}) dx d\tau,
 \end{split}
\end{equation*} \\
which completes the proof.
\end{proof}

\subsection{Bounds in terms of the relative energy}

\begin{lemma} \label{lemmaJ1}
 Under the conditions of Proposition \ref{relativeentropy},
 $$\mathcal{J}_1(t) \leq C \int_0^t \Psi(\tau)d\tau, \ \ \ t \in [0,T[, $$
 for some positive constant $C$.
\end{lemma}
\begin{proof}
 Note that for $t \in [0,T[,$ 
 \begin{equation*}
  \begin{split}
    \mathcal{J}_1(t) &=  -\varepsilon \int_0^t \int_\Omega \nabla \bar{u}: \rho (u - \bar{u}) \otimes (u - \bar{u})+\nabla \bar{v}: n (v - \bar{v}) \otimes (v - \bar{v})dxd \tau \\
   & \leq ( ||\nabla \bar{u} ||_{\infty} + ||\nabla \bar{v} ||_{\infty}) \int_0^t \int_\Omega \varepsilon \rho |u - \bar{u}|^2+\varepsilon  n |v - \bar{v}|^2dxd\tau \\
   & \leq C \int_0^t \Psi(\tau)d\tau.
  \end{split}
\end{equation*}
\end{proof}
\begin{lemma}
Under the conditions of Proposition \ref{relativeentropy},
 $$\mathcal{J}_2(t) \leq C \int_0^t \Psi(\tau)d\tau, \ \ \ t \in [0,T[, $$
 for some positive constant $C$.
\end{lemma}
\begin{proof}
 From conditions (\ref{pbound}) and \eqref{thermoconsistency} it follows that 
 \begin{equation*}
  \begin{split}
   p_i(r| \bar{r}) &= (r-\bar{r})^2 \int_0^1 \int_0^\tau p_i^{\prime \prime}\big(sr+(1-s)\bar{r}\big)dsd\tau \\
     & \leq (r-\bar{r})^2 \hat{k}_i  \int_0^1 \int_0^\tau h_i^{\prime \prime}\big(sr+(1-s)\bar{r}\big)dsd\tau \\
     & \leq \hat{k}_i h_i(r | \bar{r}).
  \end{split}
\end{equation*} \par
Thus, for $t \in [0,T[,$
 \begin{equation*}
  \begin{split}
    \mathcal{J}_2(t) &=  - \int_0^t \int_\Omega (\nabla \cdot \bar{u}) p_1(\rho | \bar{\rho})+(\nabla \cdot \bar{v}) p_2(n | \bar{n})dxd \tau \\
   & \leq ( ||\nabla \cdot \bar{u} ||_{\infty} + ||\nabla \cdot \bar{v} ||_{\infty})(\hat{k}_1+\hat{k}_2) \int_0^t \int_\Omega h_1(\rho | \bar{\rho})+h_2(n | \bar{n})dxd\tau \\
   & \leq C \int_0^t \Psi(\tau)d\tau.
  \end{split}
\end{equation*}
\end{proof}

\begin{lemma} \label{lemmaJ3}
 Under the conditions of Proposition \ref{relativeentropy} and for $\gamma_1, \gamma_2 \geq 2-\frac{1}{d}$,
 $$\mathcal{J}_3(t) \leq C \int_0^t \Theta(\tau)d\tau, \ \ \ t \in [0,T[, $$
 for some positive constant $C$.
\end{lemma}
\begin{proof}
 Let $\gamma = \min \{\gamma_1,\gamma_2 \}.$ The proof is divided into two cases: $\gamma \geq 2$ and $\gamma \in [2-\frac{1}{d},2[$. 
 
 \medskip
 \noindent
 \textit{Case $\gamma \geq 2$} : 
 Using inequality $ab \leq \tfrac{1}{2}a^2 + \tfrac{1}{2}b^2$ and Lemma \ref{hlemma}, one derives
  \begin{equation*}
  \begin{split}
    \mathcal{J}_3(t) &=   \int_0^t \int_\Omega \big((\rho - \bar{\rho})\bar{u}-(n - \bar{n})\bar{v}\big)\cdot \nabla (\phi - \bar{\phi})  dx d\tau \\
   & \leq ( ||\bar{u} ||_{\infty} + ||\bar{v} ||_{\infty})\int_0^t \int_\Omega |\rho - \bar{\rho}||\nabla (\phi-\bar{\phi})|+|n - \bar{n}||\nabla (\phi-\bar{\phi})| dxd\tau \\
   & \leq C \int_0^t \int_\Omega \Big(|\rho - \bar{\rho}|^2+|n - \bar{n}|^2+|\nabla (\phi -\bar{\phi})|^2 \Big) dxd\tau \\
   & \leq C \int_0^t \int_\Omega h_1(\rho| \bar{\rho})+h_2(n| \bar{n})+|\nabla (\phi-\bar{\phi})|^2dxd\tau \\
   & \leq C \int_0^t \Psi(\tau)d\tau, \qquad  t \in [0,T[.
   \end{split}
\end{equation*}

\medskip
\noindent
\textit{Case $\gamma \in [2-\frac{1}{d}, 2[$} : Fix $t \in [0,T[$ and let $q= \frac{2}{3-\gamma}, \ q'= \frac{q}{q-1},$ and 
$p=\frac{2d}{d(\gamma -1)+2},$ so that $q^\prime = \frac{dp}{d-p}$. Since $\gamma \in [2-\frac{1}{d},2[,$ then $1<p \leq q < \gamma < 2$.
Set $J(t) := \int_\Omega\big( (\rho - \bar \rho) \bar u - (n - \bar n) \bar v \big) \cdot \nabla (\phi - \bar \phi) dx$ and note that                                                                                                                                                                        
\begin{equation}\label{estim1}
  \begin{split}
     J(t)   & \leq \Big |  \int_\Omega \big((\rho - \bar{\rho})\bar{u}-(n - \bar{n})\bar{v}\big)\cdot \nabla (\phi - \bar{\phi})  dx \Big |
     \\
    & \leq  ( ||\bar{u} ||_{\infty} + ||\bar{v} ||_{\infty}) \int_\Omega (|\rho - \bar{\rho}|+|n - \bar{n}|) |\nabla (\phi - \bar{\phi})|  dx \\
   & \leq C \Big( \int_\Omega(|\rho - \bar{\rho}|+|n - \bar{n}|)^q dx \Big)^{\frac{1}{q}} \Big( \int_\Omega |\nabla (\phi - \bar{\phi})|^{q'} dx \Big)^{\frac{1}{q'}}. 
   \end{split}
\end{equation}

Consider the Neumann problem
\begin{equation*}
\begin{cases}
- \Delta (\phi - \bar \phi ) = \rho - n - \bar \rho + \bar n  & \mbox{in $\Omega$}
\\
\; \; \frac{ \partial}{\partial \nu} (\phi - \bar \phi) = 0 & \mbox{on  $\partial \Omega$}.
\end{cases}
\end{equation*}
Let 
$ f = \rho - n - \bar \rho + \bar n$ and $\varphi = \nabla (\phi -\bar{\phi}). $
Then $f \in L^\gamma(\Omega) \subseteq L^p(\Omega)$ and $\varphi = \nabla _x N * f.$ Define $\tilde{f}$ by 
 $$
 \tilde{f} = 
   \begin{cases}
    f, \ \text{in} \ \Omega
     \\
    0, \ \text{in} \ \mathbb{R}^d\setminus\Omega.
   \end{cases}
$$
Clearly $\tilde{f} \in L^p(\mathbb{R}^d),$ and from the properties of the Neumann function one deduces that 
$$|\varphi(x)| \leq C I_1(|\tilde{f}|)(x), \ \ x \in \Omega. $$ Thus, Proposition \ref{steinprop} with $\alpha = 1$ and $p = \frac{2d}{d(\gamma -1)+2}$ implies that 
\begin{equation}\label{estim2}
  \Big( \int_\Omega |\nabla (\phi - \bar{\phi})|^{q'} dx \Big)^{\frac{1}{q'}}  = ||\varphi ||_{L^{\frac{dp}{d-p}}(\Omega)} 
    \leq C ||I_1(|\tilde{f}|) ||_{L^{\frac{dp}{d-p}}(\Omega)} 
    \leq C ||f ||_{L^p(\Omega)}. 
  \end{equation}
Furthermore, choosing $r > 0$ so that $\frac{1}{r} = \frac{1}{p} - \frac{1}{q}$ it yields $$||f ||_{L^p(\Omega)}  \leq |\Omega|^\frac{1}{r} ||f ||_{L^q(\Omega)}.$$ 
Combining \eqref{estim1} and \eqref{estim2} gives
\begin{equation}\label{estim3}
  \begin{split}
    J(t) & \leq  C \Big( \int_\Omega(|\rho - \bar{\rho}|+|n - \bar{n}|)^q dx \Big)^{\frac{1}{q}} \Big( \int_\Omega(|\rho - \bar{\rho}-n + \bar{n}|)^q dx \Big)^{\frac{1}{q}}.\\
   & \leq C \Big( \int_\Omega(|\rho - \bar{\rho}|+|n - \bar{n}|)^q dx \Big)^{\frac{2}{q}} \\
   & \leq C \Big( \int_\Omega |\rho - \bar{\rho}|^qdx \Big)^{\frac{2}{q}}+C \Big( \int_\Omega |n - \bar{n}|^qdx \Big)^{\frac{2}{q}}.
   \end{split}
\end{equation}

Our next goal is to show 
\begin{equation}\label{estim4}
\Big( \int_\Omega |\rho - \bar{\rho}|^qdx \Big)^{\frac{2}{q}} \leq C \int_\Omega h_1(\rho | \bar{\rho})dx.
\end{equation}
To this end, we split the domain into
$B(t)=\{x \in \Omega \ | \ 0 \leq \rho \leq R \}$ and $U(t) = \{x \in \Omega \ | \ \rho > R \},$ where $R > M_1 + 1$ is as in Lemma \ref{hlemma}.
First observe that $$\Big( \int_\Omega |\rho - \bar{\rho}|^q dx\Big)^{\frac{2}{q}} \leq  C \Big( \int_{B(t)} |\rho - \bar{\rho}|^q dx\Big)^{\frac{2}{q}} + C \Big( \int_{U(t)} |\rho - \bar{\rho}|^q dx\Big)^{\frac{2}{q}}.$$
Since $q < 2,$ the inclusion $L^2(\Omega) \subseteq L^q(\Omega)$ holds, and together with Lemma \ref{hlemma} implies
$$\Big( \int_{B(t)} |\rho - \bar{\rho}|^q dx\Big)^{\frac{2}{q}} \leq C \int_{B(t)} |\rho - \bar{\rho}|^2 dx \leq C \int_{\Omega} h_1(\rho | \bar{\rho})dx.$$ \\
Moreover, since $\frac{1}{q}=\frac{\theta}{\gamma}+(1-\theta)$ with $ 2\theta= \gamma$ one has 
\begin{equation*}
 \begin{split}
  \Big( \int_{U(t)} |\rho - \bar{\rho}|^q dx\Big)^{\frac{2}{q}} & = \Big( \int_{U(t)} |\rho - \bar{\rho}|^{(1-\theta) q} |\rho - \bar{\rho}|^{\theta q}dx \Big)^{\frac{2}{q}} \\
  & \leq \Bigg( \Big(  \int_{U(t)} |\rho - \bar{\rho}| dx \Big)^{(1-\theta)q} \Big( \int_{U(t)} |\rho - \bar{\rho}|^{\gamma}dx \Big)^{\frac{\theta q}{\gamma}}\Bigg)^{\frac{2}{q}} \\ 
  & \leq (M+\bar M)^{2 - \gamma}  \int_{U(t)} |\rho - \bar{\rho}|^\gamma dx.
   \end{split}
\end{equation*}  
If $\gamma = \gamma_1$, then $$\int_{U(t)} |\rho - \bar{\rho}|^\gamma dx \leq C \int_\Omega h_1(\rho | \bar{\rho})dx $$ immediately follows from Lemma \ref{hlemma}. \\
If $\gamma = \gamma_2$, then $\gamma \leq \gamma_1$, and since $|\rho-\bar{\rho}| > 1$ in $U(t)$, again by Lemma \ref{hlemma} one obtains 
$$
    \int_{U(t)} |\rho - \bar{\rho}|^\gamma dx  \leq  \int_{U(t)} |\rho - \bar{\rho}|^{\gamma_1} dx \leq C \int_\Omega h_1(\rho | \bar{\rho})dx.
  $$
Consequently, we obtain \eqref{estim4} as projected.

In a similar fashion, it holds
\begin{equation}\label{estim5}
\Big( \int_\Omega |n - \bar{n}|^qdx \Big)^{\frac{2}{q}} \leq C \int_\Omega h_2(n | \bar{n})dx.
\end{equation}
Then \eqref{estim3} in conjunction with  \eqref{estim4}, \eqref{estim5} gives
$$
J(t) \leq  C \int_\Omega h_1(\rho | \bar{\rho})+h_2(n|\bar{n})dx \leq C \Psi(t),
$$
wherefrom
$$ \mathcal{J}_3(t) = \int_0^t J(\tau) d\tau \leq  C \int_0^t \Psi(\tau) d\tau, $$
which completes the proof.
\end{proof}

\begin{lemma}\label{lemmaJ4}
Under the conditions of Proposition \ref{relativeentropy},
 $$
 \mathcal{J}_4(t) \leq \frac{1}{2} \int_0^t \int_\Omega \rho|u-\bar{u}|^2+n|v-\bar{v}|^2dxd \tau + C \varepsilon^2, \ \ \  t \in [0,T[,
 $$
 for some positive constant $C$.
\end{lemma}
\begin{proof}
The boundedness of $\bar{e}_1$ and $\bar{e}_2$, the conservation of mass and $\textbf{(H)}$ imply, for $t \in [0,T[$, that
  \begin{equation*}
  \begin{split}
    \mathcal{J}_4(t) &=  - \varepsilon \int_0^t \int_\Omega \frac{\rho}{\bar{\rho}} \bar{e}_1 \cdot (u - \bar{u}) + \frac{n}{\bar{n}} \bar{e}_2 \cdot (v - \bar{v}) dx d\tau \\
   & \leq \frac{1}{2 } \int_0^t \int_\Omega \rho|u-\bar{u}|^2+n|v-\bar{v}|^2 dxd\tau+\frac{\varepsilon^2}{2}\int_0^t \int_\Omega \rho \bigg|\frac{\bar{e}_1}{\bar{\rho}}\bigg|^2+n \bigg|\frac{\bar{e}_2}{\bar{n}}\bigg|^2dxd\tau\\
   & \leq \frac{1}{2 } \int_0^t \int_\Omega \rho|u-\bar{u}|^2+n|v-\bar{v}|^2 dxd\tau + C \varepsilon^2 t\\
  \end{split}
\end{equation*}
for some positive constant $C$.
\end{proof}

Combining (\ref{relativeentropyinequality}) with the bounds in  Lemmas \ref{lemmaJ1}-\ref{lemmaJ4} gives
 \begin{equation} \label{REafterbounds}
  \Psi(t)+\frac{1}{2} \int_0^t \int_\Omega \rho|u-\bar{u}|^2+n|v-\bar{v}|^2dxd \tau \leq \Psi(0) + C \int_0^t \Psi(\tau)d\tau + C \varepsilon^2 t,
   \  \ t \in [0,T[.
 \end{equation}
Theorem \ref{mainresult} follows by the Gronwall inequality. The constant $C$ depends on $d$, $\Omega,$ $\gamma_1,$ $\gamma_2,$ 
$k_1,$ $k_2,$ $\hat{k}_1,$ $\hat{k}_2,$ $M$, $\bar M$, 
$\delta_1,$ $\delta_2,$ $M_1,$ $M_2,$ $||\bar{u}||_\infty,$ $||\bar{v}||_\infty,$ $||\nabla \bar{u}||_\infty,$ $||\nabla \bar{v}||_\infty,$ 
$||\bar{e}_1||_\infty$ and $||\bar{e}_2||_\infty.$

\appendix
\section*{Appendix}
\setcounter{equation}{0}
\renewcommand\thesection{A}

This section provides a formal derivation of system (\ref{EP2}) from the bipolar Boltzmann-Poisson model with friction terms adapted from \cite{lenard}. \par
Consider the system 
\begin{equation} \label{boltz}
\begin{dcases}
 \frac{\partial f}{\partial t}  + w \cdot \nabla f - \nabla \phi \cdot \nabla_w f  = \nabla_w \cdot \Big(\frac{1}{\tau} wf \Big) \\
  \frac{\partial g}{\partial t}  + w \cdot \nabla g + \nabla \phi \cdot \nabla_w g  = \nabla_w \cdot \Big(\frac{1}{\tau} wg \Big) \\
  - \Delta \phi = \int f dw - \int g dw
\end{dcases}
 \end{equation}
in the phase space-time $]0,+\infty[ \ \times\ \mathbb{R}^d \times \mathbb{R}^d,$ where $w \in \mathbb{R}^d$ is called the pseudo-wave vector, and $\nabla_w$ is the gradient operator with respect to $w.$ One can visualize the functions $f = f(t,x,w)$ and $g=g(t,x,w)$ as distribution functions of sets of electrons and holes, respectively. It is assumed that $f$ and $g$ satisfy 
\begin{equation} \label{boltzbc}
\lim\limits_{|w| \to + \infty} f(t,x,w) = \lim\limits_{|w| \to + \infty} g(t,x,w)= 0.
\end{equation}
Moreover, $\phi = \phi(t,x)$ is the electrostatic potential of the system and $1/\tau > 0$ is the effective collision frequency. \par
Set $$\rho = \int f dw, \ u = \frac{1}{\rho} \int wf dw ,$$
and 
$$n = \int g dw, \ v = \frac{1}{n} \int wg dw.$$
Integrating the first equation of (\ref{boltz}) with respect to the variable $w$ gives 
$$\frac{\partial}{\partial t} \int f dw+ \nabla \cdot \int w f dw = \nabla \phi \cdot \int \nabla_w f dw + \int \nabla_w \cdot \Big(\frac{1}{\tau} wf \Big) dw,$$
whence, using (\ref{boltzbc}), 
\begin{equation} \label{boltz5}
\rho_t + \nabla \cdot (\rho u) = 0. 
\end{equation}
Similarly one obtains 
\begin{equation} \label{boltz6}
n_t + \nabla \cdot (nv) = 0. 
\end{equation}
\par
Next, the momentum equations are deduced. Multiplying the first equation of (\ref{boltz}) by $w_j$ and integrating with respect to $w$ yields 
\begin{equation} \label{boltz2}
 \frac{\partial}{\partial t} \int w_j f dw+ \sum_{i=1}^d \frac{\partial}{\partial x_i} \Big( \int w_j w_i f dw \Big) = \nabla \phi \cdot \int w_j \nabla_w f dw + \int w_j \nabla_w \cdot \Big(\frac{1}{\tau} wf \Big) dw.
\end{equation}
Setting $$z_{ij} = \frac{1}{\rho} \int w_j w_i f dw,$$ one can rewrite (\ref{boltz2}) as 
\begin{equation} \label{boltz3}
 (\rho u_j)_t + \sum_{i=1}^d \frac{\partial}{\partial x_i} \big( \rho z_{ij} \big) = - \rho \frac{\partial \phi}{\partial x_j} - \frac{1}{\tau} \rho u_j.
\end{equation} \par
Defining $\sigma_{ij}=z_{ij} - u_i u_j$, after subtracting equation $$u_j \Big(\rho_t + \sum_{i=1}^d \frac{\partial}{\partial x_i} \big(\rho u_i \big) \Big) = 0 $$ from equation (\ref{boltz3}), one obtains
\begin{equation} \label{boltz4}
 \rho \frac{\partial u_j}{\partial t}  + \rho u \cdot \nabla u_j + \sum_{i=1}^d \frac{\partial}{\partial x_i} \big(\rho \sigma_{ij} \big) = - \rho \frac{\partial \phi}{\partial x_j} - \frac{1}{\tau} \rho u_j.
\end{equation} \par
Assume that $\rho \sigma_{ij}$ is a stress tensor that represents a pressure. Precisely, $\rho \sigma_{ij} = \delta_{ij} p_1(\rho)$ where $p_1$ is a function that symbolizes the pressure. Thus, $$\sum_{i=1}^d \frac{\partial}{\partial x_i} \big(\rho \sigma_{ij} \big) = \frac{\partial}{\partial x_j} \big( p_1(\rho) \big) $$
and the vector form of (\ref{boltz4}) can be written as 
\begin{equation} \label{boltz7}
 (\rho u)_t  + \nabla \cdot (\rho u \otimes u) + \nabla p_1(\rho) = - \rho \nabla \phi -\frac{1}{\tau} \rho u. 
\end{equation}
Analogously,
\begin{equation} \label{boltz8}
 (n v)_t + \nabla \cdot (n v \otimes v)+\nabla p_2(n) = n \nabla \phi - \frac{1}{\tau} nv,
\end{equation}
where again $p_2$ represents a pressure.
\par 
The third equation of (\ref{boltz}) together with equations (\ref{boltz5}), (\ref{boltz6}), (\ref{boltz7}) and (\ref{boltz8}) give system (\ref{EP2}).
\section*{Acknowledgement}
The authors wish to express their graditude to the anonymous referee for the suggestions that undoubtedly helped to improve the quality of the manuscript. The first author would also like to thank Rogerio Jorge and Xiaokai Huo for helpful discussions.

\end{document}